\documentclass{amsart}
\newtheorem{lemma}{Lemma}[section]
\newtheorem{theorem}[lemma]{Theorem}
\newtheorem{remark}[lemma]{Remark}

\newtheorem{coro}[lemma]{Corollary}
\newtheorem{definition}[lemma]{Definition}
\newtheorem{example}[lemma]{Example}
\title[Bohr/Levitan Almost Periodic and Almost Automorphic
Solutions of \ldots ]{Bohr/Levitan Almost Periodic and Almost
Automorphic Solutions of Linear Stochastic Differential Equations
without Favard's Separation Condition.}

\author{David ~Cheban}
\address[D. Cheban]{%
State University of Moldova\\ Faculty of Mathematics and
Informatics\\ Department of Mathematics  \\ A. Mateevich Street
60\\ MD--2009 Chi\c{s}in\u{a}u, Moldova} \email[D.
Cheban]{cheban@usm.md, davidcheban@yahoo.com}

\date{\today}
\subjclass{34C27, 35B15, 37B55, 60H10, 60H15.}
\keywords{Bohr/Levitan almost periodic solutions, almost
automorphic solutions, linear stochastic differential equations,
Favard theory.}

\begin{document}
%%%%%%%%%%%%%%%%%%%%%%%%%%%%%%
%%%%%%%%%%%%%%%%%%%%%%%%%%%%%%
%%%%%%%%%%

\begin{abstract}
We prove that the linear stochastic equation
$dx(t)=(A(t)x(t)+f(t))dt+g(t)dW(t)$ with linear operator $A(t)$
generating a continuous linear cocycle $\varphi$ and Bohr/Levitan
almost periodic or almost automorphic coefficients
$(A(t),f(t),g(t))$ admits a unique Bohr/Levitan almost periodic
(respectively, almost automorphic) solution in distribution sense
if it has at least one precompact solution on $\mathbb R_{+}$ and
the linear cocycle $\varphi$ is asymptotically stable.
\end{abstract}

\maketitle

Dedicated to the memory of Professor \textbf{V. V. Zhikov}.

\section{Introduction}

This paper is dedicated to the study of linear stochastic
differential equations with Bohr/Levitan almost periodic and
almost automorphic coefficients. This field is called Favard's
theory \cite{Lev-Zhi,Zh-Le}, due to the fundamental contributions
made by J. Favard \cite{Fav27}. In 1927, J. Favard published his
celebrated paper, where he studied the problem of existence of
almost periodic solutions of equation in $\mathbb{R}$ of the
following form:
\begin{equation}
\label{eq0.1}x^{\prime}=A(t)x+f(t)
\end{equation}
with the matrix $A(t)$ and vector-function $f(t)$ almost periodic
in the sense of Bohr (see, for example, \cite{Fin,Lev-Zhi}).

Along with equation (\ref{eq0.1}), consider the homogeneous
equation
\begin{equation}
x^{\prime}=A(t)x \nonumber
\end{equation}
and the corresponding family of limiting equations
\begin{equation}
x^{\prime}=B(t)x, \label{eq0.3}%
\end{equation}
where $B\in H(A),$ and $H(A)$ denotes the hull of almost periodic
matrix $A(t)$ which is composed by those functions $B(t)$ obtained
as uniform limits
on $\mathbb{R}$ of the type $B(t):=\lim\limits_{n\rightarrow\infty}%
A(t+t_{n}),$ where $\{t_{n}\}$ is some sequence in $\mathbb{R}$.

\begin{theorem}
\label{tF} (Favard's theorem \cite{Fav27}) The linear differential
equation (\ref{eq0.1}) with Bohr almost periodic coefficients
admits at least one Bohr almost periodic solution if it has a
bounded solution, and each bounded solution $\varphi(t)$ of every
limiting equation (\ref{eq0.3}) ($B\in H(A)$) is separated from
zero, i.e.
\begin{equation}\label{eqF}
\inf\limits_{t\in\mathbb{R}}|\varphi(t)|>0.
\end{equation}
\end{theorem}

This result was generalized infinite-dimensional equation in the
works of V. V. Zhikov and B. M. Levitan \cite{Zh-Le} (see also B.
M. Levitan and V. V. Zhikov \cite[ChVIII]{Lev-Zhi}).

Favard's theorem for linear differential equation with Levitan
almost periodic (respectively, almost automorphic) coefficients
was established by B. M. Levitan \cite[ChIV]{Lev_1953}
(respectively, by Lin F. \cite{Lin_1987}).

For linear stochastic differential equation Favard's theorem was
established by Liu Z. and Wang W. in \cite{LW_2015}.

In the work \cite{CC_2009} it was proved that Favard's theorem
remains true if we replace condition (\ref{eqF}) by the following:
\begin{equation}\label{eqCC}
\inf\limits_{t\to +\infty}|\varphi(t)|=0.
\end{equation}

In this paper we establish that Favard's theorem remains true for
linear stochastic differential equations under the condition
(\ref{eqCC}).

This paper is organized as follows.

In Section 2 we collect some well known facts from the theory of
dynamical systems (both autonomous and non-autonomous). Namely,
the notions of almost periodic (both in the Bohr and Levitan
sense), almost automorphic and recurrent motions; cocycle,
skew-product dynamical system, and general non-autonomous
dynamical system, comparability of motions by character of
recurrence in the sense of Shcherbakov
\cite{Sch72}-\cite{scher85}.

Section 3 is dedicated to the proof of classical Birghoff's
theorem (about existence of compact minimal set) for
non-autonomous dynamical systems.

In Section 4 we study the problem of strongly comparability of
motions by character of recurrence of semi-group non-autonomous
dynamical systems. The main result is contained in Theorem
\ref{t5.2} and it generalizes the known Shcherbakov's result
\cite{Sch72,scher85}.

Section 5 is dedicated to the shift dynamical systems and
different classes of Poisson stable functions. In particular:
quasi-periodic, Bohr almost periodic, almost automorphic functions
and many others.

In Section 6 we collect some results and constructions related
with Linear (homogeneous and nonhomogeneous) Differential Systems.
We also discusses here the relation between two definitions of
hyperbolicity (exponential dichotomy) for linear non-autonomous
systems. The main result of this section (Theorem \ref{thDP1})
establish the equivalence of two definitions for
finite-dimensional and for some classes of infinite-dimensional
systems.

Section 7 is dedicated to the study of Bohr/Levitan almost
periodic and almost automorphic solutions of Linear Stochastic
Differential Equations. The main results (Theorems \ref{thS1},
\ref{thS2} and Corollaries \ref{cor_1}, \ref{cor_11}) show that
classical Favard's theorem remains true (under some conditions)
for linear stochastic differential equations.

%\section{Nonautonomous Dynamical Systems}

\section{Cocycles, Skew-Product Dynamical Systems and Non-Autonomous
Dynamical Systems}

Let $X$ be a complete metric space, $\mathbb{R}$ $(\mathbb{Z})$ be
a group of real (integer) numbers, $\mathbb{R_{+}}$
$(\mathbb{Z_{+}})$ be a semi-group of nonnegative real (integer)
numbers, $\mathbb T$ be one of the two sets $\mathbb{R}$ or
$\mathbb{Z}$ and $\mathbb{S}\subseteq\mathbb{T}$
$(\mathbb{T_{+}}\subseteq\mathbb{S}$) be a sub-semigroup of the
additive group $\mathbb{T}$, where $\mathbb T_{+}:=\{t\in\mathbb
T:\ t\ge 0\}$.

Let $(X,\mathbb S,\pi)$ be a dynamical system.

\begin{definition}\label{def1.1}
Let $(X,\mathbb{T}_1,\pi)$ and $(Y,\mathbb{T}_2,\sigma)$ \
($\mathbb{T_{+}}\subseteq\mathbb{T}_1\subseteq\mathbb{T}_2
\subseteq\mathbb{T}$) be two dynamical systems. A mapping $h:X\to
Y$ is called a homomorphism (isomorphism, respectively) of the
dynamical system $(X,\mathbb{T}_1,\pi)$ onto
$(Y,\mathbb{T}_2,\sigma)$, if the mapping $h$ is continuous
(homeomorphic, respectively) and $h(\pi(x,t))=\sigma(h(x),t)$ (
$t\in\mathbb{T}_1,\ x\in X$). In this case the dynamical system
$(X,\mathbb{T}_1,\pi)$ is an extension of the dynamical system
$(Y,\mathbb{T}_2,\sigma)$ by the homomorphism $h$, but the
dynamical system $(Y,\mathbb{T}_2,\sigma)$ is called a factor of
the dynamical system\index{factor of dynamical system}
$(X,\mathbb{T}_1,\pi)$ by the homomorphism $h$. The dynamical
system $(Y,\mathbb{T}_2,\sigma)$ is called also a base of the
extension\index{base of extension} $(X,\mathbb{T}_1,\pi)$.
\end{definition}

\begin{definition}\label{def1.2}
A triplet $\langle (X,\mathbb{T}_1,\pi),\,(Y,\mathbb{T}_2,\sigma),
\,h\rangle $, where $h$ is a homomorphism from
$(X,\mathbb{T}_1,\pi)$ onto $(Y,\mathbb{T}_2,\sigma),$ is called a
non-autonomous dynamical system (NDS).
\end{definition}

\begin{definition}\label{def1.3}
A triplet $\langle W, \varphi, (Y,\mathbb{T}_2,\sigma)\rangle $ (or
shortly $\varphi$), where $(Y,\mathbb{T}_2,\sigma)$ is a dynamical
system on $Y$, $W$ is a complete metric space and $\varphi$ is a
continuous mapping from $\mathbb{T}_1\times W\times Y$ to $W$,
satisfying the following conditions:
\begin{enumerate}
\item[a.]
$\varphi(0,u,y)=u$ $(u\in W, y\in Y)$;
\item[b.]
$\varphi(t+\tau,u,y)= \varphi(\tau,\varphi(t,u,y),\sigma(t,y))$
$(t,\tau\in\mathbb{T}_1,\, u\in W, y\in Y),$
\end{enumerate}
is called \cite{Sel_71} a cocycle on $(Y,\mathbb{T}_2,\sigma)$
with the fiber $W$.
\end{definition}

\begin{definition}\label{def1.4}
Let $X:= W\times Y$ and define a mapping $\pi: X\times
\mathbb{T}_1\to X$ as following:
$\pi((u,y),t):=(\varphi(t,u,y),\sigma(t,y))$ (i.e.,
$\pi=(\varphi,\sigma)$). Then it is easy to see that
$(X,\mathbb{T}_1,\pi)$ is a dynamical system on $X$ which is
called a skew-product dynamical system \cite{Sel_71} and
$h=pr_2:X\to Y$ is a homomorphism from $(X,\mathbb{T}_1,\pi)$ onto
$(Y,\mathbb{T}_2,\sigma)$ and, consequently, $\langle
(X,\mathbb{T}_1,\pi),\, (Y,\mathbb{T}_2,\sigma), h\rangle $ is a
non-autonomous dynamical system.
\end{definition}

Thus, if we have a cocycle $\langle W, \varphi, (Y,\mathbb{T}_2,
\sigma)\rangle $ on the dynamical system $(Y,\mathbb{T}_2,\sigma)$
with the fiber $W$, then it generates a non-autonomous dynamical
system $\langle (X,\mathbb{T}_1,\pi),$\ $(Y,\mathbb{T}_2,\sigma),
h\rangle $ ($X:= W\times Y$) called a non-autonomous dynamical
system generated by the cocycle\index{non-autonomous dynamical
system generated by cocycle} $\langle W, \varphi,
(Y,\mathbb{T}_2,\sigma)\rangle $ on $(Y,\mathbb{T}_2,\sigma)$.

Non-autonomous dynamical systems (cocycles) play a very important
role in the study of non-autonomous evolutionary differential
equations. Under appropriate assumptions every non-autonomous
differential equation generates a cocycle (a non-autonomous
dynamical system). Below we give some examples of theses.

\begin{example}\label{ex3.1.5}
{\rm Let $E$ be a real or complex Banach space and $Y$ be a metric
space. Denote by $C(Y\times E,E)$ the space of all continuous
mappings $f:Y \times E\mapsto E$ endowed by compact-open topology.
Consider the system of differential equations
\begin{equation}\label{eq3.1.9}
\left\{\begin{array}{ll}u'=&F(y,u)\\
                        y '=&G(y) ,
\end{array}
\right.
\end{equation}
where $Y\subseteq E, G\in C(Y,E)$ and $F\in C(Y\times E,E)$.
Suppose that for the system (\ref{eq3.1.9}) the conditions of the
existence, uniqueness, continuous dependence of initial data and
extendability on $\mathbb{R}_{+}$ are fulfilled. Denote by
$(Y,\mathbb{R}_{+},\sigma)$ a dynamical system on $Y$ generated by
the second equation of the system (\ref{eq3.1.9}) and by
$\varphi(t,u,y)$ -- the solution of equation
\begin{equation}\label{eq3.1.10}
u'=F(y t,u) \ (y t:=\sigma(t,y))
\end{equation}
passing through the point $u\in E$ for $t=0$. Then the mapping
$\varphi:\mathbb{R}_{+}\times E\times Y\to E$ is continuous and
satisfies the conditions: $\varphi(0,u,y)=u$ and
$\varphi(t+\tau,u,y)= \varphi(t,\varphi(\tau,u,y),y t)$ for all
$t,\tau \in \mathbb R_{+}$, $u\in E$ and $y\in Y$ and,
consequently, the system (\ref{eq3.1.9}) generates a
non-autonomous dynamical system $\langle
(X,\mathbb{R}_{+},\pi),(Y,\mathbb{R}_{+},\sigma),h\rangle $ (where
$X:=E\times Y$, $\pi:=(\varphi,\sigma)$ and $h:=pr_2:X\to Y$).

We will give some generalization of the system (\ref{eq3.1.9}).
Namely, let $(Y,\mathbb{R}_{+},\sigma)$ be a dynamical system on
the metric space $Y$. Consider the system
\begin{equation}\label{eq3.1.11}
\left \{\begin{array}{ll}u'=&F(y t,u)\\
                       y\in & Y ,
\end{array}
\right.
\end{equation}
where $F\in C(Y\times E,E)$. Suppose that for the equation
(\ref{eq3.1.10}) the conditions of the existence, uniqueness and
extendability on $\mathbb{R}_{+}$ are fulfilled. The system
$\langle (X,\mathbb{R}_{+},\pi),(Y,$
$\mathbb{R}_{+},\sigma),h\rangle $, where $X:=E\times Y$,
$\pi:=(\varphi,\sigma)$, $\varphi(\cdot,u,y)$ is the solution of
(\ref{eq3.1.10}) and $h:=pr_2:X\to Y$ is a non-autonomous
dynamical system generated by the equation (\ref{eq3.1.11}).}
\end{example}

\begin{example}\label{ex1.1}
{\rm Let us consider a differential equation
\begin{equation}\label{eq1.0.6}
u'=f(t,u),
\end{equation}
where $f\in C(\mathbb{R}\times E,E)$. Along with equation
$(\ref{eq1.0.6})$ we consider its $H$-class
\cite{Bro79},\cite{Lev-Zhi}, \cite{Sel_71}, \cite{scher85}, i.e.,
the family of equations
\begin{equation}
v'=g(t,v),\label{eq1.0.7}
\end{equation}
where $g\in H(f):=\overline{\{f^{\tau}:\tau\in \mathbb{R}\}}$,
$f^{\tau}(t,u):=f(t+\tau,u)$ for all $(t,u)\in \mathbb{R}\times E$
and by bar we denote the closure in $C(\mathbb{R}\times E,E)$. We
will suppose also that the function $f$ is regular, i.e. for every
equation (\ref{eq1.0.7}) the conditions of the existence,
uniqueness and extendability on $\mathbb{R}_{+}$ are fulfilled.
Denote by $\varphi(\cdot,v,g)$ the solution of equation
(\ref{eq1.0.7}) passing through the point $v\in E$ at the initial
moment $t=0$. Then there is a correctly defined mapping
$\varphi:\mathbb{R}_{+}\times E\times H(f)\to E$ satisfying the
following conditions (see, for example, \cite{Bro79},
\cite{Sel_71}):
\begin{enumerate}
\item[$1)$] $\varphi(0,v,g)=v$ for all $v\in E$ and $g\in H(f)$;
\item[$2)$]
$\varphi(t,\varphi(\tau,v,g),g^{\tau})=\varphi(t+\tau,v,g)$ for
every $ v\in E$, $g\in H(f)$ and $t,\tau \in \mathbb{R}_{+}$;
\item[$3)$] the mapping $\varphi:\mathbb{R}_{+}\times E\times
H(f)\to E$ is continuous.
\end{enumerate}

Denote by $Y:=H(f)$ and $(Y,\mathbb{R}_{+},\sigma)$ a dynamical
system of translations (a semigroup system) on $Y$, induced by the
dynamical system of translations $(C(\mathbb{R}\times
E,E),\mathbb{R},\sigma)$. The triplet $\langle E,\varphi,
(Y,\mathbb{R}_{+},\sigma)\rangle $ is a cocycle on
$(Y,\mathbb{R}_{+},\sigma)$ with the fiber $E$. Thus, equation
(\ref{eq1.0.6}) generates a cocycle $\langle E,\varphi,
(Y,\mathbb{R}_{+},\sigma)\rangle $ and a non-autonomous dynamical
system $\langle (X,\mathbb{R}_{+},\pi),\, (Y,\mathbb{R}_{+},\sigma),
h\rangle $, where $X:= E\times Y$, $\pi:=(\varphi,\sigma)$ and
$h:=pr_2:X\to Y$.}
\end{example}

\begin{remark}
Let $Y:=H(f)$ and $(Y,\mathbb R,\pi)$ be the shift dynamical
system on $Y$. The equation (\ref{eq1.0.6}) (the family of
equation (\ref{eq1.0.7})) may be written in the form
(\ref{eq3.1.10}), where $F:Y\times E\mapsto E$ is defined by
equality $F(g,u):=g(0,u)$ for all $g\in H(f)=Y$ and $u\in E,$ then
$F(g^{t},u)=g(t,u)$ ($g^{t}(s,u):=\sigma(t,g)(s,u)=g(t+s,u)$ for
all $t,s\in \mathbb R$ and $u\in E$).
\end{remark}

\subsection{Recurrent, Almost Periodic and Almost Automorphic Motions.}

Let $(X,\mathbb S,\pi)$ be a dynamical system.

\begin{definition}\label{def4.2}
A number $\tau \in \mathbb S$ is called an $\varepsilon >0$ shift
of $x$ (respectively, almost period of $x$), if $\rho
(x\tau,x)<\varepsilon $ (respectively, $\rho (x(\tau
+t),xt)<\varepsilon$ for all $t\in \mathbb S$).
\end{definition}

\begin{definition}\label{def14.3}
A point $x \in X $ is called almost recurrent (respectively, Bohr
almost periodic), if for any $\varepsilon >0$ there exists a
positive number $l$ such that at any segment of length $l$ there is
an $\varepsilon$ shift (respectively, almost period) of point $x\in
X$.
\end{definition}

\begin{definition}\label{def4.4}
If the point $x\in X$ is almost recurrent and the set
$H(x):=\overline{\{xt\ \vert \ t\in \mathbb S\}}$ is compact, then
$x$ is called recurrent.
\end{definition}

Denote by $\mathfrak N_{x}:=\{\{t_n\}:\ \{t_n\}\subset \mathbb S \
\mbox{such that}\ \{\pi(t_n,x)\}\to x\ \mbox{as}\ n\to \infty\}$.

\begin{definition}
A point $x\in X$ of the dynamical system $(X,\mathbb S, \pi)$ is
called Levitan almost periodic \cite{Lev-Zhi}, if there exists a
dynamical system $(Y,\mathbb S,\sigma)$ and a Bohr almost periodic
point $y\in Y$ such that $\mathfrak N_{y}\subseteq \mathfrak
N_{x}.$
\end{definition}

\begin{definition} A point $x\in X$ is called stable in the sense
of Lagrange $(st.$L$)$, if its trajectory
$\Sigma_{x}:=\Phi\{\pi(t,x)\ :\ t\in \mathbb S\}$ is relatively
compact.
\end{definition}

\begin{definition} A point $x\in X$ is called almost automorphic
in the dynamical system $(X,\mathbb S,\pi),$ if the following
conditions hold: \begin{enumerate}\item $x$ is st.$L$; \item the
point $x\in X$ is Levitan almost periodic.
\end{enumerate}
\end{definition}

\begin{lemma}\label{l4.5*}\cite{CM04}
Let $(X,\mathbb S,\pi)$ and $(Y,\mathbb S,\sigma)$ be two
dynamical systems, $x\in X$ and the following conditions be
fulfilled:
\begin{enumerate} \item a point $y\in Y$ is Levitan almost periodic;
\item $\mathfrak N_y\subseteq \mathfrak N_x$.
\end{enumerate}
Then the point $x$ is Levitan almost periodic, too.
\end{lemma}

\begin{coro}\label{cor4.6*} Let $x\in X$ be a st.$L$ point, $y\in
Y$ be an almost automorphic point and $\mathfrak N_y\subseteq
\mathfrak N_x$. Then the point $x$ is almost automorphic too.
\end{coro}
\begin{proof}
Let $y$ be an almost automorphic point, then by Lemma \ref{l4.5*}
the point $x\in X$ is Levitan almost periodic. Since $x$ is
st.$L$, then it is almost automorphic.
\end{proof}

\begin{remark}\label{r4.1*}
We note (see, for example, \cite{Lev-Zhi} and \cite{scher85}) that
if $y \in Y $ is a stationary ($\tau$-periodic, almost periodic,
quasi periodic, recurrent) point of the dynamical system $(Y
,\mathbb T_{2} ,\sigma)$ and $h:Y \to X $ is a homomorphism of the
dynamical system $(Y,\mathbb T _{2},\sigma)$ onto $(X,\mathbb
T_{1},\pi)$, then the point $x=h(y)$ is a stationary
($\tau$-periodic, almost periodic, quasi periodic, recurrent)
point of the system $(X,\mathbb T_{1},\pi)$.
\end{remark}

\begin{definition}\label{defPC1} A point $x_0\in X$ is called
\cite{scher85,sib}
\begin{enumerate}
\item[-] pseudo recurrent if for any $\varepsilon >0$, $t_0\in\mathbb
T$ and $p\in \Sigma_{x_0}$ there exist numbers
$L=L(\varepsilon,t_0)>0$ and $\tau =\tau(\varepsilon,t_0,p)\in
[t_0,t_0+L]$ such that $\tau \in \mathfrak T (p,\varepsilon))$;
\item[-] pseudo periodic (or uniformly Poisson stable) if for any $\varepsilon >0$, $t_0\in\mathbb
T$ there exists a number  $\tau =\tau(\varepsilon,t_0)>t_0$ such
that $\tau \in \mathfrak T (p,\varepsilon))$ for any $p\in
\Sigma_{x_0}$.
\end{enumerate}
\end{definition}

\begin{remark}\label{remPR_1} 1. Every pseudo periodic point is
pseudo recurrent.

2. If $x\in X$ is pseudo recurrent, then
\begin{enumerate}
\item[-] it is Poisson stable; \item[-] every point $p\in H(x)$ is
pseudo recurrent; \item[-] there exist pseudo recurrent points for
which the set $H(x_0)$ is compact but not minimal
\cite[ChV]{Sch72}; \item[-] there exist pseudo recurrent points
which are not almost automorphic (respectively, pseudo periodic)
\cite[ChV]{Sch72}.
\end{enumerate}
\end{remark}

\subsection{Comparability of Motions by the Character
of Recurrence}

In this subsection following B. A. Shcherbakov
\cite{Sch75,scher85} (see also \cite{Che_1977},
\cite[ChI]{Che_2009}) we introduce the notion of comparability of
motions of dynamical system by the character of their recurrence.
While studying stable in the sense of Poisson motions this notion
plays the very important role (see, for example,
\cite{Sch72,scher85}).

Let $(X,\mathbb S,\pi)$ and $(Y,\mathbb S,\sigma)$ be dynamical
systems, $x\in X$ and $y\in Y$. Denote by
$\Sigma_{x}:=\{\pi(t,x):\ t\in\mathbb S\}$ and $\mathfrak
M_{x}:=\{\{t_n\}:$ such that $\{\pi(t_n,x)\}$ converges as $n\to
\infty \}$.

\begin{definition}
A point $x_0\in X$ is called
\begin{enumerate}
\item[a.] comparable by the character of recurrence with $y_0\in Y$ if
there exists a continuous mapping $h:\Sigma_{y_0}\mapsto
\Sigma_{x_0}$ satisfying the condition
\begin{equation}\label{eqC1_1}
h(\sigma(t,y_0))=\pi(t,x_0)\ \ \mbox{for any}\ t\in \mathbb R ;
\end{equation}
\item[b.] strongly comparable by the character of recurrence with $y_0\in Y$ if
there exists a continuous mapping $h:H(y_0)\mapsto H(x_0)$
satisfying the condition
\begin{equation}\label{eqC2_1}
h(y_0)=x_0 \ \mbox{and}\ h(\sigma(t,y))=\pi(t,h(x))\ \mbox{for
any}\ y\in H(x_0) \ \mbox{and}\ t\in \mathbb R ;
\end{equation}
\item[c.] uniformly comparable by the character of recurrence with
$y_0\in Y$ if there exists a uniformly continuous mapping
$h:\Sigma_{y_0}\mapsto \Sigma_{x_0}$ satisfying condition
(\ref{eqC1_1}).
\end{enumerate}
\end{definition}

\begin{theorem}\label{thC1} Let $x_0\in X$ be uniformly
comparable by the character of recurrence with $y_0\in Y$. If the
spaces $X$ and $Y$ are complete, then $x_0$ is strongly comparable
by the character of recurrence with $y_0\in Y$.
\end{theorem}
\begin{proof} Let $h:\Sigma_{y_0}\mapsto \Sigma_{x_0}$ be a
uniformly continuous mapping satisfying condition (\ref{eqC1_1})
and the spaces $X$ and $Y$ be complete. Then $h$ admits a unique
continuous extension $h: H(y_0)\to H(x_0)$. Now we will show that
this map possesses property (\ref{eqC2_1}). Tho this end we note
that by condition $h$ satisfies equality (\ref{eqC1_1}). Let now
$y\in H(y_0)$ and $t\in\mathbb R$, then there exists a sequence
$\{t_n\}\subset \mathbb T$ such that $\sigma(t_n,y_0)\to y$ as
$n\to \infty$ and, consequently, $\sigma(t+t_n,y_0)\to
\sigma(t,y)$. Since the sequence $\{\sigma(t_n,y_0)\}$ is
convergent and the map $h:\Sigma_{y_0}\mapsto \Sigma_{x_0}$ is
uniformly continuous, satisfies (\ref{eqC1_1}) and the spaces $X$
and $Y$ are complete, then the sequence
$\{\pi(t_n,x_0)\}=\{h(\sigma(t_n,y_0))\}$ is also convergent.
Denote by $x:=\lim\limits_{n\to \infty}\pi(t_n,x_0)$. Then we have
\begin{eqnarray}\label{eqC4}
& h(\sigma(t,y))=\lim\limits_{n\to
\infty}h(\sigma(t_n+t,y_0))=\lim\limits_{n\to
\infty}\pi(t_n+t,x_0)= \nonumber  \\
& \lim\limits_{n\to
\infty}\pi(t,\pi(t_n,x_0))=\pi(t,x)=\pi(t,h(y)).   \nonumber
\end{eqnarray}
Theorem is proved.
\end{proof}

\begin{coro}\label{corC1} The uniform comparability implies
strong comparability (if the phase spaces are complete) and strong
comparability implies the (simple) comparability.
\end{coro}

\begin{theorem}\label{thC2}\cite{Che_1977},\cite[ChI]{Che_2009}
Let $X$ and $Y$ be two complete metric
spaces, then the following statement are equivalent:
\begin{enumerate}
\item the point $x_0$ is strongly comparable by the character of recurrence with $y_0\in
Y$;
\item $\mathfrak M_{y_0}\subseteq \mathfrak M_{x_0}$.
\end{enumerate}
\end{theorem}

\begin{theorem}\label{thC3}\cite{Sch75} If the spaces $X$ and $Y$
are complete and $y_0$ is Lagrange stable, then the strong
comparability implies uniform comparability and, consequently, they
are equivalent.
\end{theorem}

\begin{remark}\label{r2.1.6}
From Theorems $\ref{thC1}$ and $\ref{thC3}$ follows that the strong
comparability of the point $x_0$ with $y_0$ is equivalent to their
uniform comparability if the point $y_0$ is st. $L$ and the phase
space $X$ and $Y$ are complete. In general case these notions are
apparently different $($though we do not know the according
example$)$.
\end{remark}

\begin{theorem}\label{thPC1} \cite{Sch72,scher85}  Let $x_0\in X$ be uniformly
comparable by the character of recurrence with $y_0\in Y$. If
$y_0\in Y$ is pseudo recurrent (respectively, pseudo periodic), then
$x_0$ is so.
\end{theorem}
\begin{proof} If $y_0$ is
pseudo recurrent (respectively, pseudo periodic) then for any
$\varepsilon >0$, $t_0\in\mathbb T$ and $p\in \Sigma_{x_0}$ there
exist $L=L(\varepsilon,t_0)>0$ and $\tau =\tau
(\varepsilon,t_0,p)\in [t_0,t_0+L]$ (respectively, there exists
$\tau =\tau (\varepsilon,t_0)>t_0$) such that $\tau \in \mathfrak
T (p,\varepsilon)$ for any $p\in \Sigma_{x_0}$. Since $x_0$ is
uniformly comparable by the character of recurrence with $y_0$,
then there exists a uniformly continuous mapping
$h:\Sigma_{y_0}\mapsto \Sigma_{x_0}$ satisfying (\ref{eqC1_1}).
Let $p\in \Sigma_{x_0}$, $q\in h^{-1}(p)$, $\delta =\delta
(\varepsilon)>0$ be chosen from the uniform continuity of $h$ and
$\tilde{L}(\varepsilon,t_0):=L(\delta(\varepsilon),t_0)>0$ and
$\tilde{\tau}(\varepsilon,t_0,p):=\tau(\delta(\varepsilon),t_0,q)\in
[t_0,t_0+\tilde{L}(\varepsilon,t_0)]$ (respectively,
$\tilde{\tau}(\varepsilon,t_0):=\tau(\delta(\varepsilon),t_0)>t_0$),
then $\tau\in \mathfrak T (p,\varepsilon)$ because $\tau\in
\mathfrak T(q,\delta)$ and $h(q)=p$. Theorem is proved.
\end{proof}

\begin{theorem}\label{thPC2} Let $\langle (X,\mathbb T,\pi), (Y,\mathbb T,\sigma),h
\rangle$ be a nonautonomous dynamical system and $x_0\in X$ be a
conditionally Lagrange stable point (i.e., the set $\Sigma_{x_0}$
is conditionally precompact), then the following statement hold:
\begin{enumerate}
\item if $H(x_0)\bigcap X_{y_0}$ consists a single point $\{x_0\}$,
where $y_0:=h(x_0)$, then $\mathfrak N_{y_0}\subseteq \mathfrak
N_{x_0}$;
\item if the set $H(x_0)\bigcap X_{q}$ contains at most one point  for any
$q\in H(y_0)$, then $\mathfrak M_{y_0}\subseteq \mathfrak M_{x_0}$.
\end{enumerate}
\end{theorem}
\begin{proof} Let $\{t_n\}\in \mathfrak N_{y_0}$, then
$\sigma(t_n,y_0)\to y_0$ as $n\to \infty$. Since $\Sigma_{x_0}$ is
conditionally precompact and $\{\pi(t_n,x_0)\}=\Sigma_{x_0}\bigcap
h^{-1}(\{\sigma(t_n,y_0)\})$, then $\{\pi(t_n,x_0)\}$ is a
precompact sequence. Tho show that $\{t_n\}\in\mathfrak N_{x_0}$ it
is sufficient to prove that the sequence $\{\pi(t_n,x_0)\}$ has at
most one limiting point. Let $p_{i}$ ($i=1,2$) be two limiting
points of $\{\pi(t_n,x_0)\}$, then there are
$\{t_{k_{n}^{i}}\}\subseteq \{t_n\}$ such that
$p_{i}:=\lim\limits_{n\to \infty}\pi(t_{k_{n}^{i}},x_0)$ ($i=1,2$).
Since $\{t_{k_{n}^{i}}\}\in \mathfrak N_{y_0}$, then $p_{i}\in
H(x_0)\bigcap X_{y_0}=\{x_0\}$ ($i=1,2$) and, consequently,
$p_1=p_2=x_0$. Thus we have $\mathfrak N_{y_0}\subseteq \mathfrak
N_{x_0}$.

Let now $\{t_n\}\in\mathfrak M_{y_0}$, $q\in H(y_0)$ such that
$q=\lim\limits_{n\to \infty}\sigma(t_n,y_0)$ and $H(x_0)\bigcap
X_{q}$ contains at most one point. By the same arguments as above
the sequence $\{\pi(t_n,x_0)\}$ is precompact. To show that the
sequence $\{\pi(t_n,x_0)\}$ converges we will use the similar
reasoning as above. Let $p_{i}$ ($i=1,2$) be two limiting points of
$\{\pi(t_n,x_0)\}$, then there are $\{t_{k_{n}^{i}}\}\subseteq
\{t_n\}$ such that $p_{i}:=\lim\limits_{n\to
\infty}\pi(t_{k_{n}^{i}},x_0)$ ($i=1,2$). Since
$\sigma(t_{k_{n}^{i}},y_0)\to q$ as $n\to \infty$, then $p_{i}\in
H(x_0)\bigcap X_{q}$ ($i=1,2$) and, consequently, $p_1=p_2$. Thus we
have $\mathfrak M_{y_0}\subseteq \mathfrak M_{x_0}$. Theorem is
completely proved.
\end{proof}

\begin{theorem}\label{thPC3} Let $\langle (X,\mathbb T,\pi), (Y,\mathbb T,\sigma),h
\rangle$ be a nonautonomous dynamical system and $x_0\in X$ be a
conditionally Lagrange stable point. Suppose that the following
conditions are fulfilled:
\begin{enumerate}
\item[a.] $Y$ is minimal; \item[b.] $H(x_0)\bigcap X_{y_0}$
consists a single point $\{x_0\}$, where $y_0:=h(x_0)$; \item[c.]
the set $H(x_0)$ is distal, i.e., $\inf\limits_{t\in\mathbb
T}\rho(\pi(t,x_1),\pi(t,x_2))>0$ for any $x_1,x_2\in H(x_0)$ with
$x_1\not= x_2$ and $h(x_1)=h(x_2)$.
\end{enumerate}
Then $\mathfrak M_{y_0}\subseteq \mathfrak M_{x_0}$.
\end{theorem}
\begin{proof} By Theorem \ref{thPC2} to prove this statement it is
sufficient to show that $H(x_0)\bigcap X_{q}$ contains at most one
point for any $q\in H(y_0)$. If we suppose that it is not true, then
there are points $q_0\in H(y_0)$ and $p_{0}^{i}\in H(x_0)\bigcap
X_{q_0}$ such that $p_{0}^{1}\not= p_{0}^{2}$. Then by condition c.
there exists a number $\alpha =\alpha(p_0^{1},p_{0}^{2})>0$ such
that
\begin{equation}\label{eqPC1}
\rho(\pi(t,p_{0}^{1}),\pi(t,p_{0}^{2})\ge \alpha
\end{equation}
for any $t\in \mathbb T$. Since $Y$ is minimal, then
$H(q_0)=H(y_0)=Y$ and, consequently, there exists a sequence
$\{t_n\}\subset \mathbb T$ such that $\sigma(t_n,q_0)\to y_0$ as
$n\to \infty$. Consider the sequences $\{\pi(t_n,p_{0}^{i})\}$
($i=1,2$). Since $\Sigma_{x_0}$ is conditionally precompact, then
by the same arguments as in Theorem \ref{thPC2} the sequences
$\{\pi(tc_n,p_{0}^{i})\}$ ($i=1,2$) are precompact too. Without
loss of generality we may suppose that they are convergent. Denote
by $x^{i}:=\lim\limits_{n\to \infty}\pi(t_n,p_{0}^{i})$ ($i=1,2$).
Then $x^{i}\in H(x_0)\bigcap X_{y_0}=\{x_0\}$ ($i=1,2$) and,
consequently, $x^{1}=x^{2}=x_0$. By the other hand according to
inequality $(\ref{eqPC1})$ we have $\rho(x^{1},x^{2})\ge \alpha
>0$. The obtained contradiction show that our assumption is falls,
i.e., under the conditions of Theorem $H(x_0)\bigcap X_{q}$
contains at most one point. Theorem is proved.
\end{proof}

\begin{remark}\label{remShc} Note that Theorems \ref{thPC2} and
\ref{thPC3} coincide with the results of B. A. Shcherbakov
\cite[ChIII]{Sch72} (see also \cite[ChIII]{scher85}) when the
point $x_0$ is Lagrange stable.
\end{remark}

\section{Birghoff's theorem for non-autonomous dynamical systems (NDS)}

Let $X,Y$ be two complete metric spaces and
\begin{equation}\label{eqB1}
\langle (X,\mathbb T_1,\pi), (Y,\mathbb T_2,\sigma),h)
\end{equation}
be a non-autonomous dynamical system.

\begin{definition}\label{defB1} A subset $M$ of $X$ is said to be
a minimal set of NDS (\ref{eqB1}) if it possesses the following
properties:
\begin{enumerate}
\item[a.] $h(M)=Y$; \item[b.] $M$ is positively invariant, i.e.,
$\pi(t,M)\subseteq M$ for any $t\in\mathbb T_1$; \item[c.] $M$ is
a minimal subset of $X$ possessing properties a. and b..
\end{enumerate}
\end{definition}

\begin{remark} 1. In the case of autonomous dynamical systems
(i.e., $Y$ consists a single point) the definition above coincides
with the usual notion of minimality.

2. If the NDS $\langle (X,\mathbb T_1,\pi), (Y,\mathbb
T_2,\sigma),h)$ is periodic (i.e., there exists a $\tau$-periodic
point $y_0\in Y$ such that $Y=\{\sigma(t,y_0):\ t\in [0,\tau)\}$),
then the nonempty compact set $M\subset X$ is a minimal set of NDS
(\ref{eqB1}) if and only if the set $M_{y_0}:=X_{y_0}\bigcap M$ is
a minimal set of the discrete (autonomous) dynamical system
$(X_{y_0},P)$ generated by positive powers of the map
$P:=\pi(\tau,\cdot):X_{y_0}\to X_{y_0}$.
\end{remark}

\begin{lemma}\label{lB1}  Let $M\subset X$ be a nonempty, closed and positively invariant subset of NDS (\ref{eqB1}) such that $h(M)=Y$,
then the following statements hold:
\begin{enumerate}
\item if $H(x)=M$ for any $x\in M$, where
$H(x):=\overline{\{\pi(t,x):\ t\in \mathbb T_1\}}$, then $M$ is a
minimal set of NDS (\ref{eqB1}); \item if
\begin{enumerate}
\item $\mathbb T_{1}=\mathbb T_{2}$;
\item $Y$ is a minimal set of
autonomous dynamical system $(Y,\mathbb T_2,\sigma)$;
\item $M$ is
a minimal subset of NDS (\ref{eqB1}) and it is conditionally
compact,
\end{enumerate}
then $H(x)=M$ for any $x\in M$.
\end{enumerate}
\end{lemma}
\begin{proof} Let $M\subset X$ be a nonempty, closed and positively invariant subset of NDS
(\ref{eqB1}) such that $h(M)=Y$ and $H(x)=M$ for any $x\in M$. We
will show that in this case $M$ is a minimal set of NDS
(\ref{eqB1}). If we suppose that it is not true, then there exists
a subset $\tilde{M}\subset M$ such that $\tilde{M}$ is a nonempty,
closed, positively invariant, $h(\tilde{M})=Y$ and $\tilde{M}\not=
M$. Let $x\in \tilde{M}\subset M$, then by condition of Lemma we
have $M=H(x)\subseteq \tilde{M}\subset M$ and, consequently,
$M=\tilde{M}$. The obtained contradiction proves our statement.

Suppose that $M$ is a minimal subset of NDS (\ref{eqB1}) and it is
conditionally compact. We will establish that, then $H(x)=M$ for
any $x\in M$. In fact. If it is not so, then there exists a point
$x_0\in M$ such that $H(x_0)\subset M$ and $H(x_0)\not= M$. Since
$h(\pi(t,x_0))=\sigma(t,y_0)$ (where $y_0:=h(x_0)$) for any
$t\in\mathbb T_{1}$ and $Y$ is minimal, then for any $y\in Y$
there exists a sequence $\{t_n\}\subset \mathbb T$ such that
$\sigma(t_n,y_0)\to y$ as $n\to \infty$. Since the set $M$ is
conditionally compact without loss of generality we can suppose
that the sequence $\{\pi(t_n,x_0)\}$ converges. Denote by
$x:=\lim\limits_{n\to \infty}\pi(t_n,x_0)$, then we have $h(x)=y$.
Since $y\in Y$ is an arbitrary point, then $H(H(x_0))=Y$. Thus we
have a nonempty, positively invariant subset $H(x_0)\subset M$
such that $h(H(x_0))=Y$ and $H(x_0)\not= M$. This fact contradicts
to the minimality of $M$. The obtained contradiction proves the
second statement of Lemma.
\end{proof}

\begin{coro}\label{corB1} Let $\langle (X,\mathbb T,\pi), (Y,\mathbb
T,\sigma),h)$ be a non-autonomous dynamical system and $M\subset
X$ be a nonempty, conditionally compact and positively invariant
set. If the dynamical system $(Y,\mathbb T,\sigma)$ is minimal,
then the subset $M$ is a minimal subset of NDS (\ref{eqB1}) if and
only if $H(x)=M$ for any $x\in M$.
\end{coro}

\begin{theorem}\label{thB1} Suppose that $\langle (X,\mathbb T_1,\pi), (Y,\mathbb
T_2,\sigma),h)$ is a non-autonomous dynamical system and $X$ is
conditionally compact, then there exists a minimal subset $M$.
\end{theorem}
\begin{proof}
Denote by $\mathfrak A (X)$ the family of all nonempty, positively
invariant and conditionally compact subsets $A\subseteq X$. Note
that $\mathfrak A (X)\not=\emptyset$ because $X\in \mathfrak A
(X)$. It is clear that the family $\mathfrak A (X)$ partially
ordered with respect to inclusion $\subseteq$. Namely: $A_1\le
A_2$ if and only if $A_1\subseteq A_2$ for all $A_1,A_2\in
\mathfrak A (X)$. If $\mathcal A\subseteq\mathfrak A (X)$ is a
linear ordered subfamily of $\mathfrak A (X)$, then the
intersection $M$ of subsets of the family $\mathcal A$ is
nonempty. In fact. For any $y\in Y$ the family of subsets
$\mathcal A_{y}:=\{A_{y}:\ A\in \mathcal A\}$, where
$A_{y}:=A\bigcap X_{y}$, is linear ordered and, consequently,
$$
M_{y}=\bigcap \{A_{y}:\ A\in \mathcal A\}\not= \emptyset
$$
because $X_{y}$ is compact. Thus $M$ is a closed, positively
invariant set such that $h(M)=Y$ and, consequently, $M\in
\mathfrak A(X)$. By Lemma of Zorn the family $\mathfrak A(X)$
contains at least one minimal element $M$. It is clear that $M$ is
a minimal set. Theorem is proved.
\end{proof}

If $X$ is a compact metric space, then $X^{X}$ denote the
collection of all maps from $X$ to itself, provided with the
product topology, or, what is the same thing, the topology of
pointwise convergence. By Tikhonov theorem, $X^{X}$ is compact.

$X^{X}$ has a semigroup structure defined by the composition of
maps.

Let $ \langle (X,\mathbb T _{1},\pi),(Y,\mathbb
T_{2},\sigma),h\rangle $ be a non-autonomous dynamical system and
$y \in Y$ be a Poisson stable point. Denote by
$${E}_{y}^{+}:=\{\xi |\quad \exists \{t_{n}\}\in \mathfrak N
_{y}^{+\infty} \quad \mbox{such that}\quad \pi ^{t_{n}}|_{X_{y}}\to \xi\},$$
where $X_{y}:=\{x\in X|\quad h(x)=y \}$ and $\to$ means the
pointwise convergence\index{${E}_{y}$} and $\mathfrak N_{x}^{+\infty}:=
\{\{t_n\}\in \mathfrak N_{x}:\ t_n\to +\infty \ \mbox{as}\ n\to \infty \}$.

\begin{lemma} \label{l109.2.4}\cite[ChIX,XV]{Che_2015}
Let $y \in Y$ be a Poisson stable point, $\langle (X,\mathbb
T_{1},\pi ),$ $(Y,\mathbb T_2,\sigma),h\rangle $ be a
non-autonomous dynamical system and $X$ be a conditionally compact
set. Then $E_{y}^{+} $ is a nonempty compact sub-semigroup of the
semigroup $X_{y}^{X_{y}}.$
\end{lemma}

\begin{lemma}\label{l10G3}\cite[ChIX,XV]{Che_2015}
Let $X$ be a conditionally compact metric space and
$\langle (X,\mathbb T_{1},\pi ),$ $(Y,\mathbb T_2,\sigma),h\rangle
$ be a non-autonomous dynamical system. Suppose that the following
conditions are fulfilled:
\begin{enumerate}
\item The point $y \in Y$ is Poisson stable; \item $\lim
\limits_{t\to +\infty}\rho (\pi(t,x_1),\pi(t,x_2))=0$ for all
$x_1,x_2\in X_{y}:=h^{-1}(y)=\{x\in X\ :\ h(x)=y \}.$
\end{enumerate}

Then there exists a unique point $x_{y}\in X_{y}$ such that
$\xi(x_{y})=x_{y}$ for all $\xi \in  E_{y}^{+}.$
\end{lemma}

\begin{remark}\label{r10G2} 1. If a point $x\in X$ is compatible by the character of
the recurrence with $y \in Y$ and $y$ is a stationary
(respectively, $\tau$-periodic, recurrent, Poisson stable) point,
then the point $x$ is so \cite{scher85}.

2. If a point $x\in X$ is strongly compatible by the character of
the recurrence with $y \in Y$ and $y$ is a stationary
(respectively, $\tau$-periodic, quasi periodic, Bohr almost
periodic, almost automorphic, recurrent in the sense of Birkhoff,
Levitan almost periodic, almost recurrent, strongly Poisson stable
and $H(y)$ is minimal, Poisson stable) point, then the point $x$
is so \cite{scher85}.
\end{remark}

\begin{coro}\label{cor10G1}  Let $X$ be a conditionally compact metric space and
$\langle (X,\mathbb T_{1},\pi ),$ $(Y,\mathbb T_2,\sigma),h\rangle
$ be a non-autonomous dynamical system. Suppose that the following
conditions are fulfilled:
\begin{enumerate}
\item The point $y \in Y$ is Poisson stable; \item $\lim
\limits_{t\to +\infty}\rho (\pi(t,x_1),\pi(t,x_2))=0$ for all
$x_1,x_2\in X_{y}:=h^{-1}(y)=\{x\in X\ :\ h(x)=y \}.$
\end{enumerate}
Then there exists a unique point $x_{y} \in X_{y}$ which is
compatible by the character of the recurrence with $y \in Y$ point
$x_{y}\in X_{y}$ such that
$$
\lim\limits_{t\to +\infty}\rho(\pi(t,x),\pi(t,x_{y}))=0
$$
for all $x\in X_{y}.$
\end{coro}

\begin{coro}\label{cor10G2} Let $y\in Y$ be a stationary
$($respectively, $\tau$-periodic, recurrent, Poisson stable$)$
point. Then under the conditions of Corollary $\ref{cor10G1}$
there exists a unique stationary $($respectively, $\tau$-periodic,
recurrent, Poisson stable$)$ point $x_{y}\in X_{y}$ such that
$$
\lim\limits_{t\to +\infty}\rho(\pi(t,x),\pi(t,x_{y}))=0
$$
for all $x\in X_{y}.$
\end{coro}

Let $x_0\in X$. Denote by $\omega_{x_0}:=\bigcap\limits_{t\ge
0}\overline{\bigcup\{\pi(\tau,x_0):\ \tau \ge t\}}$ the
$\omega$-limit set of $x_0$.

\begin{lemma}\label{l3.12}
Let $\langle (X,\mathbb T_{+} ,\pi),$ $(Y,\mathbb T,\sigma)$
$,h\rangle $ be a nonautonomous dynamical system and
$\Sigma_{x_0}^{+}:=\{\pi(t,x_0):\ t\ge 0\}$ be conditionally
precompact. Then for any $x\in \omega_{x_0}$ there exists at least
one entire trajectory of dynamical system $(X,\mathbb T_{+} ,\pi
)$ passing through the point $x$ for $t=0$ and $ \gamma (\mathbb
T\Omega)\subseteq \omega_{x_0}\ ( \gamma (\mathbb T):=\{\gamma
(t)|\ t\in \mathbb T \})$.
\end{lemma}
\begin{proof}
Let $ x\in \omega_{x_0}$, then there are $\{t_n\}\subset \mathbb T
$ such that $x=\lim \limits_{n \to \infty}\pi(t_n,x_0)$ and $t_n
\to +\infty $ as $n\to \infty$. We consider the sequence $
\{\gamma_n \}\subset C(\mathbb T,X)$ defined by equality
$$ \gamma_n (t)=\pi(t+t_n,x_0),\quad \mbox{if}\quad t\ge -t_n \quad
\mbox{and} \quad \gamma _n (t)=x_0 \quad \mbox{for}\quad  t\le
-t_n.$$

Let $l$ be an arbitrary positive number. We will prove that the
sequence $\{\gamma _n\}$ is equicontinuous on segment
$[-l,l]\subset \mathbb T.$ If we suppose that it is not true, then
there exist $\varepsilon _0, l_0>0, t^i_n \in [-l_0,l_0]$ and
$\delta _n \to 0 \quad (\delta _n
>0)$ such that
\begin{equation}\label{eq9.2.3}
\vert t^{1}_n-t^{2}_n\vert \le \delta _n \quad \mbox{and} \quad
\rho (\gamma _n (t^{1}_n ), \gamma _n (t^{2}_n ))\ge \varepsilon
_0.
\end{equation}
We may suppose that $t_n^{i}\to t_0 \quad (i=1,2).$ Since $t_n\to
+\infty$ as $n\to \infty$, then there exists a number $n_0\in
\mathbb N$ such that $t_n\ge l$ for any $n\ge n_0$. From
(\ref{eq9.2.3}) we obtain
\begin{equation}\label{eq9.2.4}
\varepsilon _0 \le \rho (\gamma _n (t_n^{1}),\gamma _n
(t_n^{2}))=\rho (
\pi(t_n^{1}+l_0,\pi(t_n-l_0,x_0)),\pi(t_n^{2}+l_0,\pi(t_n-l_0,x_0)))
\end{equation}
for any $ n \ge n_0$. Note that
$h(\pi(t_n-l_0,x_0))=\sigma(t_n-l_0,y_0)\to \sigma(y_0,-l_0)$ as
$n\to \infty$. Since $\Sigma_{x_0}^{+}$ is conditionally
precompact, then the sequence $\{\pi(t_n-l_0,x_0)\}$ is relatively
compact. Without loss of generality we can suppose that
$\{\pi(t_n-l_0,x_n)\}$ converges and denote by $\bar{x}$ its
limit. Passing into limit in (\ref{eq9.2.4}) as $n\to \infty$ and
taking into account above we obtain
\begin{equation}\label{eqC1_10}
\varepsilon _0 \le \rho (
\pi(t_0+l_0,\bar{x}),\pi(t_0+l_0,\bar{x}))=0.\nonumber
\end{equation}
The obtained contradiction proves our statement.

Now we will prove that the set $\{\gamma_n(t):\ t\in [-l,l],\ n\in
\mathbb N\}$ is precompact. To this end we note that for any $n\ge
n_0$ we have
$h(\gamma_n(t))=h(\pi(t+t_n,x_0))=\sigma(t,\sigma(t_n,y_0))$ and,
consequently, the set $K:=\{\sigma(t,\sigma(t_n,y_0)):\ t\in
[-l,l]\}\subset Y$ is precompact. Since the set $\Sigma_{x_0}^{+}$
is conditionally precompact and $\{\gamma_n(t):\ t\in
[-l,l]\}\subset h^{-1}(K)\bigcap \Sigma_{x_0}^{+}$, then the set
$\{\gamma_n(t):\ t\in [-l,l]\}$ is also precompact and,
consequently,  $\{\gamma _n\}$ is a relatively compact sequence of
$C(\mathbb T,X)$ since $\{\gamma_{n}\}$ is equicontinuous on
$[-l,l]$.

Let $\gamma $ be a limiting point of the sequence $\{\gamma _n\}$,
then there exists a subsequence $ \{\gamma _{k _n}\}$ such that $
\gamma (t)=\lim \limits _{n \to \infty} \gamma _{k_n} (t) $
uniformly on every segment $[-l,l]\subset \mathbb T.$ In
particular $ \gamma \in C(\mathbb T,X)$ and $\gamma(t)\in
\omega_{x_0}$ for any $t\in \mathbb T$ because
$\gamma(t)=\lim\limits_{n\to \infty}\pi(t+t_n,x_0)$. We note that
$$
\pi ^t\gamma(s)=\lim \limits _{n \to \infty}\pi^t\gamma _{k_n}(s)=
\lim \limits _{n \to \infty}\gamma _{k_n}(s+t)=\gamma (s+t)
$$
for all $t\in \mathbb T_{+}$ and $s\in \mathbb T$. Finally, we see
that $\gamma (0)=\lim \limits _{n \to \infty}\gamma _{k_n}(0)=
\lim \limits _{n\to \infty}\pi^{t_{k_n}}x_{0}=x,$ i.e., $\gamma$
is an entire trajectory of dynamical system $(X,\mathbb
T_{+},\pi)$ passing through point $x$. The Lemma is completely
proved.
\end{proof}

\section{Semi-group Dynamical Systems}

Let
$\langle(X,\mathbb{T}_{+},\pi),(Y,\mathbb{T},\lambda),h\rangle$
(respectively, $\langle W,\varphi,(Y,\mathbb{T},\lambda)\rangle$,
where $\varphi:\mathbb{T}_{+}\times W\times Y\mapsto W$) be a
semi-group non-autonomous dynamical system (respectively, a
semi-group cocycle), where $\mathbb{T}_{+}:=\{t\in\mathbb{T}\ |\
t\geq0\}$.

A continuous mapping $\gamma:\mathbb{T}\mapsto X$ (respectively,
$\nu:\mathbb{T}\mapsto W$) is called an entire trajectory of the
semi-group dynamical system $(X,\mathbb{T}_{+},\pi)$
(respectively, semi-group cocycle $\langle W,
\varphi,(Y,\mathbb{T},\lambda)\rangle$ or shortly $\varphi$)
passing through the point $x$ (respectively, $(u,y)$), if
$\gamma(0)=x$ (respectively, $\nu(0)=u$) and
$\pi(t,\gamma(s))=\gamma(t+s)$ (respectively,
$\varphi(t,\nu(s),\lambda(s,y))=\nu(t+s)$) for all
$t\in\mathbb{T}_{+}$ and $s\in\mathbb{T}$.

The entire trajectory $\gamma$ of the semigroup dynamical system
$(X,\mathbb{T}_{+},\pi)$ is said to be comparable with $y\in Y$ by
the character of recurrence ($(Y,\mathbb{T},\lambda)$ is a
two-sided
dynamical system) if $\mathfrak{N}_{y}\subseteq\mathfrak{N}_{\gamma}%
$, where $\mathfrak{N}_{\gamma}:=\{\{t_{n}\}\subset\mathbb{R}\ |\
$ such that the sequence $\{\gamma(t+t_{n})\}$ converges uniformly
with respect to $t$ on
every compact from $\mathbb{T}$, i.e. it converges in the space $C(\mathbb{T}%
,X)$.

\begin{remark}
\label{r5.21} Let
$\langle(X,\mathbb{T}_{+},\pi),\,(Y,\mathbb{T},\lambda ),h\rangle$
be a semi-group non-autonomous dynamical system, $M$ be subset of
$X$. Denote by $\tilde{M}:=\{x\in M\ |\ $
there exists at least one entire trajectory $\gamma$ of $(X,\mathbb{T}_{+}%
,\pi)$ passing through the point $x$ with condition $\gamma(\mathbb{T}%
)\subseteq M \}$. It is easy to see that the set $\tilde{M}$ is
invariant, i.e. $\pi(t,\tilde{M})=\tilde{M}$ for all
$t\in\mathbb{T}_{+}.$ Moreover, $\tilde{M}$ is the maximal
invariant set which is contained in $M$.
\end{remark}

Denote by $\Phi(M)$ the family of all entire trajectories $\gamma$
of a semi-group dynamical system $(X,\mathbb{T}_{+},\pi)$ with
condition $\gamma(\mathbb{T})\subseteq M$.

Let $\langle(X,\mathbb{T}_{+},\pi),\,(Y,\mathbb{T},\lambda
),h\rangle$ be a semi-group non-autonomous dynamical system,
$x_0\in X$, $y_0:=h(x_0)$ and $y\in \omega_{y_0}$. Denote by
$\Phi_{y}$ the family of all full trajectory $\gamma$ of
semi-group dynamical system $(X,\mathbb T_{+},\pi)$ satisfying the
condition: $h(\gamma(0))=y$ and $\gamma(\mathbb T)\subseteq
\omega_{x_0}$.

\begin{lemma}\cite{CC_2009}
\label{l5.1*} Assume that $\langle(X,\mathbb{T}_{+},\pi),(Y
,\mathbb{T},\lambda),h\rangle$ is a semi-group non-auto\-no\-mous
dynamical
system, $M$ is a conditionally compact subset of $X,$ and $\tilde{M}%
\not =\emptyset$. Then, the following statements hold:

\begin{enumerate}
\item the set $\tilde{M}$ is closed;

\item $\tilde{Y}:=h(\tilde{M})$ is a closed and invariant subset
of $(Y,\mathbb{T},\lambda);$

\item $\Phi(M)$ is a closed and shift invariant subset of
$C(\mathbb{T},M)$ and, consequently, on $\Phi(M)$ is induced a
shift dynamical system
$(\Phi(M),\mathbb{T},\sigma)$ by Bebutov's dynamical system $(C(\mathbb{T}%
,M),\mathbb{T},\sigma)$;

\item the mapping $H:\Phi(M)\mapsto\tilde{Y}$ defined by equality
$H(\gamma):=h(\gamma(0))$ is a homomorphism of the dynamical
system $(\Phi(M),\mathbb{T},\sigma)$ onto
$(\tilde{\Omega},\mathbb{T},\lambda)$, i.e. the map $H$ is
continuous and
\begin{equation}
\label{eq5.2*}H(\sigma(t,\gamma))=\lambda(t,H(\gamma))\nonumber
\end{equation}
for all $\gamma\in\Phi(M)$ and $t\in\mathbb{T}$;

\item the set $\Phi(M)$ is conditionally compact with respect to
$(\Phi(M),\tilde{Y},H)$;

\item if $\gamma_{1},\gamma_{2}\in\Phi(M)$ and
$h(\gamma_{1}(0))=h(\gamma _{2}(0))$, then the following
conditions are equivalent:

\begin{enumerate}
\item[a.]
$\lim\limits_{t\to+\infty}\rho(\gamma_{1}(t),\gamma_{2}(t))=0,$
where $\rho$ is the distance on $X$;

\item[b.] $\lim\limits_{t\to+\infty}d(\sigma(t,\gamma_{1}),\sigma
(t,\gamma_{2}))=0,$ where $d$ is the Bebutov's distance on
$\Phi(M)$.
\end{enumerate}
\end{enumerate}
\end{lemma}

\begin{theorem}\label{t5.1}\cite{CC_2009}
Let $\langle(X,\mathbb{T}_{+},\pi),$ $(Y,\mathbb{T}%
,\lambda),h\rangle$ be a semi-group non-autonomous dynamical
system. Suppose that the following conditions are fulfilled:

\begin{enumerate}
\item[a.] \textit{there exists a point } $x_{0}\in X$\textit{ such that }%
$H^{+}(x_{0})$\textit{ is conditionally compact;}

\item[b.] \textit{the point }$y_{0}:=h(x_{0})\in Y$\textit{ is
Poisson stable;}

\item[c.]
\begin{equation}\label{eq5.5}
\lim\limits_{t\rightarrow+\infty}\rho(\gamma_{1}(t),\gamma_{2}(t))=0
\nonumber%
\end{equation}
\textit{for all entire trajectories }$\gamma_{1}$\textit{ and }$\gamma_{2}%
$\textit{ of the semi-group dynamical system }$(X,\mathbb
T_{+},\pi)$\textit{ with the
conditions: }$\gamma_{i}(\mathbb T)\subseteq H^{+}(x_{0})$\textit{ and }%
$h(\gamma_{1}(0))=h(\gamma_{2}(0))=y_{0}$\textit{.}
\end{enumerate}

\textit{Then, there exists a unique entire trajectory }$\gamma\in
\Phi_{y_0}$\textit{ of }$(X,\mathbb T_{+},\pi)$\textit{ possessing
the following properties:}

\begin{enumerate}
\item $\gamma(\mathbb{T})\subseteq H^{+}(x_{0})$;

\item $h(\gamma(0))=y_{0}$;

\item $\gamma$ \textit{is comparable with} $y_{0}\in Y$ \textit{by
the character of recurrence}.
\end{enumerate}
\end{theorem}

\begin{coro}\label{cor5.1**}\cite{CC_2009}
Assume the conditions of Theorem \ref{t5.1} hold and that
\begin{equation}
\lim\limits_{t\rightarrow+\infty}\rho(\gamma_{1}(t),\gamma_{2}(t))=0
\nonumber\label{eq5.51}%
\end{equation}
for all entire trajectories $\gamma_{1}$ and $\gamma_{2}$ of the
semi-group dynamical system $(X,\mathbb{T}_{+},$ $\pi)$ with the
conditions: $\gamma
_{i}(\mathbb{T})$ (i=1,2) is conditionally compact and $h(\gamma_{1}(0))=$ $h(\gamma_{2}%
(0))=$ $y_{0}$.

Then, there exists a unique entire trajectory $\gamma$ of $(X,\mathbb{T}%
_{+},\pi),$ which is comparable with $y_{0}\in Y$ by the character
of recurrence, and which satisfies the following properties:

\begin{enumerate}
\item $\gamma(\mathbb{T})$ is conditionally compact;

\item $h(\gamma(0))=y_{0}$.
\end{enumerate}
\end{coro}

\begin{coro}\label{cor3.2*}\cite{CC_2009}
Let $y_{0}\in Y$ be a stationary ($\tau$-periodic, almost
automorphic, almost recurrent, Levitan almost periodic, Poisson
stable) point. Then under the conditions of Theorem \ref{t5.1}
there exists a unique stationary ($\tau$-periodic, almost
automorphic, almost recurrent, Levitan almost periodic, Poisson
stable) entire trajectory $\gamma$ of dynamical system
$(X,\mathbb{T}_{+},\pi)$ such that $\gamma(\mathbb{T})\subseteq
H^{+}(x_{0}).$
\end{coro}

\begin{remark}\label{remC1} Theorem \ref{t5.1} and Corollaries
\ref{cor5.1**} and \ref{cor3.2*} remain true if we replace
condition (\ref{eq5.5}) by equality:
\begin{equation}\label{eqC10}
\lim\limits_{n\to
\infty}\rho(\gamma_1(t_n),\gamma_2(t_n))=0\nonumber
\end{equation}
for any $\gamma_1,\gamma_2\in \Phi_{y_0}$ and $\{t_n\}\in
\mathfrak N_{y_0}^{+\infty}$.
\end{remark}

\begin{theorem}\label{t5.1*}\cite{CC_2009}
Let $\langle(X,\mathbb{T}_{+},\pi),$ $(Y,\mathbb{T}%
,\lambda),h\rangle$ be a semi-group non-autonomous dynamical
system. Suppose that the following conditions are fulfilled:

\begin{enumerate}
\item[a.] \textit{there exists a point } $x_{0}\in X$\textit{ such that }%
$H^{+}(x_{0})$\textit{ is conditionally compact;}

\item[b.] \textit{the point }$y_{0}:=h(x_{0})\in Y$\textit{ is
Poisson stable;}

\item[c.]
\begin{equation}\label{eq5.5*}
\lim\limits_{t\rightarrow+\infty}\rho(\pi(t,x_1),\pi(t,x_2))=0
\nonumber%
\end{equation}
for any $x_1,x_2\in X_{y_0}$.
\end{enumerate}

\textit{Then, there exists a unique entire trajectory }$\gamma\in
\Phi_{y_0}$\textit{ of }$(X,\mathbb T_{+},\pi)$\textit{ possessing
the following properties:}

\begin{enumerate}
\item $\gamma(\mathbb{T})\subseteq H^{+}(x_{0})$;

\item $h(\gamma(0))=y_{0}$;

\item $\gamma$ \textit{is comparable with} $y_{0}\in Y$ \textit{by
the character of recurrence} and
\begin{equation}\label{eq5.5**}
\lim\limits_{t\rightarrow+\infty}\rho(\pi(t,x),\gamma(t))=0
\end{equation}
for any $x\in X_{y_0}$.
\end{enumerate}
\end{theorem}

\begin{coro}\label{cor3.2_I}\cite{CC_2009}
Let $y_{0}\in Y$ be a stationary ($\tau$-periodic, almost
automorphic, almost recurrent, Levitan almost periodic, Poisson
stable) point. Then under the conditions of Theorem \ref{t5.1*}
there exists a unique stationary ($\tau$-periodic, almost
automorphic, almost recurrent, Levitan almost periodic, Poisson
stable) entire trajectory $\gamma$ of dynamical system
$(X,\mathbb{T}_{+},\pi)$ such that $\gamma(\mathbb{T})\subseteq
H^{+}(x_{0})$ and equality (\ref{eq5.5**}) takes place.
\end{coro}

The entire trajectory $\gamma$ of the semi-group dynamical system
$(X,\mathbb{T}_{+},\pi)$ is said to be strongly comparable with
the point $y$ of group dynamical system $(Y,\mathbb{T},\lambda)$
by the
character of recurrence, if $\mathfrak{M}_{y}\subseteq\mathfrak{M}%
_{\gamma},$ where
$\mathfrak{M}_{\gamma}:=\{\{t_{n}\}\subset\mathbb{T}\ |\ $ the
sequence $\sigma(t_{n},\gamma)$ converges in the space
$C(\mathbb{T},X)\}$.

A point $x\in X$ is said to be strongly Poisson stable if each
point $p\in H(x)$ is Poisson stable.

\begin{theorem}\label{thB2}
Let $X$ be a conditionally compact metric space and $\langle
(X,\mathbb T_{1},\pi ),$ $(Y,\mathbb T_2,\sigma),h\rangle $ be a
non-autonomous dynamical system. Suppose that the following
conditions are fulfilled: \begin{enumerate} \item the dynamical
system $(Y,\mathbb T,\sigma)$ is minimal; \item the point $y \in
Y$ is strongly Poisson stable; \item
\begin{equation}\label{eqB}
\lim \limits_{t\to +\infty}\rho (\pi(t,x_1),\pi(t,x_2))=0
\end{equation}
for all $x_1,x_2\in X$ such that $h(x_1)=h(x_2).$
\end{enumerate}

Then there exists a unique point $x_{y} \in X_{y}$ which is
strongly compatible by the character of the recurrence with $y \in
Y$ and
$$
\lim\limits_{t\to +\infty}\rho(\pi(t,x),\pi(t,x_{y}))=0
$$
for any $x\in X_{y}.$
\end{theorem}
\begin{proof} By Lemma \ref{l10G3}  there exists a unique fixed point
$\tilde{x}_{y}\in X_{y}$ of the semigroup $E_{y}^{+}.$ By
Corollary \ref{cor10G1}    the point $\tilde{x}_{y}$ is a unique
point in $\tilde{M}$ comparable by character of recurrence with
the point $y$. Let $\tilde{M}:=\overline{\{\pi(t,x_{y})\ : t\in
\mathbb T_{1}\}}$ then it is a conditionally compact positively
invariant set and taking into account the minimality of $Y$ using
the same argument as in the proof of Lemma \ref{lB1} we have
$H(\tilde{M})=Y$. By Theorem \ref{thB1} there exists a minimal set
$M\subset \tilde{M}$. By Corollary \ref{cor10G1} there exists a
point $x_{y}\in M$ which is a unique point in $M$ comparable by
character of recurrence with the point $y$ and, consequently,
$x_{y}$ coincides with the point $\tilde{x}_{y}$. We will show
that $M_{q}:=M\cap X_q$ (for all $q\in
H(y):=\overline{\{\sigma(t,y)\ : t\in \mathbb T\}}$) consists a
single point. If we suppose that it is not true then there exist
$q_0\in H(y)$ and $x_1,x_2\in M_{q_0}$ such that $x_1\not=x_2.$ By
Corollary \ref{cor10G1} there exists a unique point $x_{q_0}\in
M_{q_0}$ which is compatible by the character of recurrence with
the point $q_0.$ Without loss of generality we may suppose that
$x_{q_0}=x_1.$ Since the set $M$ is minimal, then there exists a
sequence $\{t_n\}\in \mathfrak N_{q_0}^{+\infty}$ such that
$\{\pi(t_n,x_1)\}\to x_2.$ On the other hand taking into
consideration the inclusion $\mathfrak N_{q_0}\subseteq \mathfrak
N_{x_1}$ we have $\{\pi(t_n,x_1)\}\to x_1$ and, consequently,
$x_1=x_2.$ The obtained contradiction prove our statement. Now to
finish the proof of Theorem it is sufficient to apply Theorem
\ref{thPC2}.
\end{proof}

\begin{coro}\label{cor10B} Let $y\in Y$ be a stationary
$($respectively, $\tau$--periodic, almost periodic, recurrent,
strongly Poisson stable and $H(y_0)$ is a minimal set$)$ point.
Then under the conditions of Theorem \ref{thB2} there exists a
unique stationary $($respectively, $\tau$--periodic, almost
periodic, recurrent, strongly Poisson stable and $H(y_0)$ is a
minimal set$)$ point $x_{y}\in X_{y}$ such that
$$
\lim\limits_{t\to +\infty}\rho(\pi(t,x),\pi(t,x_{y}))=0
$$
for all $x\in X_{y}.$
\end{coro}
\begin{proof}
This statement directly follows from Theorem \ref{thB2} and Remark
\ref{r10G2}.
\end{proof}

\begin{theorem}
\label{t5.2} Let $\langle(X,\mathbb{T}_{+},\pi),$ $(Y,\mathbb{T}%
,\lambda),h\rangle$ be a semi-group non-autonomous dynamical
system. Suppose that the following conditions are fulfilled:

\begin{enumerate}
\item[a)] there exists a point $x_{0}\in X$ such that
$H^{+}(x_{0})$ is conditionally compact;

\item[b)] the point $y_{0}:=h(x_{0})\in Y$ is strongly Poisson
stable;

\item[c)] the set $H(y_0)$ is minimal;

\item[d)]
\begin{equation}\label{eq5.52}
\lim\limits_{t\rightarrow+\infty}\rho(\gamma_{1}(t),\gamma_{2}(t))=0
\end{equation}
for all entire trajectories $\gamma_{1}$ and $\gamma_{2}$ of the
semi-group dynamical system $(X,\mathbb{T}_{+},\pi)$ with the
conditions: $\gamma _{i}(\mathbb{T})\subseteq H^{+}(x_{0})$ and
$h(\gamma_{1}(0))=h(\gamma_{2}(0))$.
\end{enumerate}

Then there exists a unique entire trajectory $\gamma$ of $(X,\mathbb{T}%
_{+},\pi)$ possessing the following properties:

\begin{enumerate}
\item $\gamma(\mathbb{T})\subseteq H^{+}(x_{0})$;

\item $h(\gamma(0))=y_{0}$;

\item $\gamma$ is strongly comparable by the character of
recurrence with $y_{0} \in Y$.
\end{enumerate}
\end{theorem}

\begin{proof}
Let $M:=H^{+}(x_{0})$. Then, by Lemma \ref{l3.12} we have that
$\Phi (M)\not =\emptyset$. Consider the group non-autonomous
dynamical system
$\langle(\Phi(M),\mathbb{T},\sigma),(\tilde{Y},\mathbb{T},\lambda
),H\rangle$ (see Lemma \ref{l5.1*}). By Lemma \ref{l3.12} the
point $y_{0}$ belongs to $\tilde{Y}$. According to Lemma
\ref{l5.1*} all conditions of Theorem \ref{thB2} are fulfilled
and, consequently, we obtain the existence of at least one entire
trajectory $\gamma$ of the dynamical system
$(X,\mathbb{T}_{+},\pi)$ which is strongly comparable with
$y_{0}\in Y$ by the character of recurrence, and $\gamma
(\mathbb{T})\subseteq H^{+}(x_{0})$. To finish the proof it is
sufficient to
show that there exists at most one entire trajectory of $(X,\mathbb{T}_{+}%
,\pi)$ with the properties (i)-(iii). Let $\gamma_{1}$ and
$\gamma_{2}$ be two entire trajectories satisfying (i)-(iii). In
particular, $\gamma
_{i}(\mathbb{T})\subseteq H^{+}(x_{0})$ and $\mathfrak{M}_{y_{0}%
}\subseteq\mathfrak{M}_{\gamma_{i}}$ ($i=1,2$). Then we also have
$\mathfrak{N}_{y_{0}}\subseteq\mathfrak{N}_{\gamma_{i}}$
($i=1,2$). From assumption c) we obtain
\begin{equation}
\lim\limits_{t\rightarrow+\infty}d(\sigma(t,\gamma_{1}),\sigma(t,\gamma
_{2}))=0. \nonumber\label{eq5.61}%
\end{equation}
On the other hand, there exists a sequence $\{t_{n}\}\in\mathfrak{N}%
_{y_{0}}\subseteq\mathfrak{N}_{\gamma_{i}}$ ($i=1,2$) such that
$t_{n}\rightarrow+\infty$ and, consequently,
\[
d(\gamma_{1},\gamma_{2})=\lim\limits_{n\rightarrow+\infty}d(\sigma
(t_{n},\gamma_{1}),\sigma(t_{n},\gamma_{2}))=0,
\]
i.e., $\gamma_{1}=\gamma_{2},$ and the proof is completed.
\end{proof}

\begin{coro}
\label{cor5.2**} In addition to assumptions in Theorem \ref{t5.2},
suppose that
\begin{equation}
\lim\limits_{t\rightarrow+\infty}\rho(\gamma_{1}(t),\gamma_{2}(t))=0
\nonumber\label{eq5.53}%
\end{equation}
for all entire trajectories $\gamma_{1}$ and $\gamma_{2}$ of the
semi-group dynamical system $(X,\mathbb{T}_{+},\pi)$ with the
conditions: $\gamma
_{i}(\mathbb{T})$ (i=1,2) is conditionally compact and $h(\gamma_{1}(0))=h(\gamma_{2}%
(0))$.

Then there exists a unique entire trajectory $\gamma\in\Phi_{y_0}$ of $(X,\mathbb{T}%
_{+},\pi)$, which is strongly comparable with $y_{0}\in Y$ by the
character of recurrence, and such that $\gamma(\mathbb{T})$ is
conditionally precompact and
$$
\lim\limits_{t\to +\infty}\rho(\pi(t,x),\gamma(t))=0
$$
for any $x\in X_{y_0}$.
\end{coro}

\begin{proof}
This statement follows by a slight modification of the proof of
Theorem \ref{t5.2}.
\end{proof}

\begin{coro}
\label{cor3.2**} Let $y_{0}\in Y$ be a stationary (respectively,
$\tau$-periodic, Bohr almost periodic, almost automorphic,
recurrent, strongly Poisson stable and $H(y_0)$ is a minimal set)
point. Then under the conditions of Theorem \ref{t5.2} ,there
exists a unique stationary (respectively, $\tau$-periodic, Bohr
almost periodic, almost automorphic, recurrent, strongly Poisson
stable and $H(y_0)$ is a minimal set) entire trajectory $\gamma$
of the dynamical system $(X,\mathbb{T}_{+},\pi)$ such that
$\gamma(\mathbb{T})\subseteq H^{+}(x_{0}).$
\end{coro}

\begin{proof}
This statement follows easily from Theorem \ref{t5.2} and Remarks
\ref{r10G2} (item 2).
\end{proof}

\begin{coro}\label{cor3.2***} Under the conditions of Corollary
\ref{cor3.2**} if
$$
\lim\limits_{t\to +\infty}\rho(\pi(t,x_1),\pi(t,x_2))=0
$$
for any $x_1,x_2\in H^{+}(x_0)$ with $h(x_1)=h(x_2)$. Then there
exists a unique full trajectory $\gamma \in \Phi_{y_0}$
($y_0:=h(x_0)$) which is strongly comparable by character of
recurrence with the point $y_0$ and
$$
\lim\limits_{t\to +\infty}\rho(\pi(t,x),\gamma(t))=0
$$
for any $x\in X_{y_0}$.
\end{coro}

\begin{remark}\label{remC2} Theorems \ref{thB2} and \ref{t5.2} remain true if we replace
condition (\ref{eqB}) (respectively, (\ref{eq5.52})) by equality:
\begin{equation}\label{eqC2}
\lim\limits_{n\to \infty}\rho(\gamma_1(t_n),\gamma_2(t_n))=0
\end{equation}
for any $\gamma_1,\gamma_2\in \Phi_{y}$, $\{t_n\}\in \mathfrak
N_{y}^{+\infty}$ and $y\in H(y_0)$.
\end{remark}

\section{Bohr/Levitan almost periodic, almost automorphic and Poisson stable functions}

Let $(X,\rho)$ be a compete metric space. Denote by $C(\mathbb R,X)$
the family of all continuous functions $f:\mathbb R\mapsto X$
equipped with the distance
\begin{equation}\label{eqD1}
d(f,g):=\sup\limits_{l>0}d_{l}(f,g),\nonumber
\end{equation}
where $d_{l}(f,g):=\min\{\max\limits_{|t|\le l}\rho(f(t),g(t));
l^{-1}\}$. The metric $d$ is complete and it defines on $C(\mathbb
R,X)$ the compact-open topology. Let $h\in\mathbb R$ denote by
$f^{h}$ the $h$-translation of $f$, that is, $f^{h}(s):=f(s+h)$
for all $s\in \mathbb R$.

Let us recall the types of Poisson stable functions to be studied
in this paper; we refer the reader to
\cite{Sel_71,Sch72,scher85,sib} and the references therein.

\begin{definition} \rm
A function $\varphi\in C(\mathbb R,X)$ is called {\em stationary}
(respectively, {\em $\tau$-periodic}) if $\varphi(t)=\varphi(0)$
(respectively, $\varphi(t+\tau)=\varphi(t)$) for all $t\in \mathbb
R$.
\end{definition}

\begin{definition} \rm
Let $\varepsilon >0$. A number $\tau \in \mathbb R$ is called {\em
$\varepsilon$-almost period} of the function $\varphi$ if
$\rho(\varphi(t+\tau),\varphi(t))<\varepsilon$ for any
$t\in\mathbb R$.
\end{definition}

Denote by $\mathcal T(\varphi,\varepsilon):=\{\tau \in\mathbb R:\
\rho(\varphi(t+\tau),\varphi(t))<\varepsilon \ \mbox{for any}\
t\in\mathbb R\}$ the set of $\varepsilon$-almost periods of
$\varphi$.

\begin{definition} \rm
A function $\varphi \in C(\mathbb R,X)$ is said to be {\em Bohr
almost periodic} if the set of $\varepsilon$-almost periods of
$\varphi$ is {\em relatively dense} for each $\varepsilon >0$,
i.e., for each $\varepsilon >0$ there exists $l=l(\varepsilon)>0$
such that $\mathcal T(\varphi,\varepsilon)\cap
[a,a+l]\not=\emptyset$ for all $a\in\mathbb R$.
\end{definition}

\begin{definition} \label{ppf}\rm %(\cite[p.32]{Bohr_I1947})
A function $\varphi \in C(\mathbb R,X)$ is said to be {\em
pseudo-periodic} in the positive (respectively, negative) direction
if for each $\varepsilon >0$ and $l>0$ there exists a
$\varepsilon$-almost period $\tau >l$ (respectively, $\tau <-l$) of
the function $\varphi$. The function $\varphi$ is called
pseudo-periodic if it is pseudo-periodic in both directions.
\end{definition}

\begin{remark} \rm
A function $\varphi \in C(\mathbb R,X)$ is pseudo-periodic in the
positive (respectively, negative) direction if and only if there
is a sequence $t_n\to +\infty$ (respectively, $t_n\to -\infty$)
such that $\varphi^{t_n}$ converges to $\varphi$ uniformly in
$t\in \mathbb R$ as $n\to \infty$.
\end{remark}

\begin{definition}\rm
The {\em hull of $\varphi$}, denoted by $H(\varphi)$, is the set
of all the limits of $\varphi^{h_n}$ in $C(\mathbb R, X)$, i.e.
\[
H(\varphi):=\{\psi\in C(\mathbb R, X): \psi=\lim_{n\to\infty}
\varphi^{h_n} \hbox{ for some sequence } \{h_n\} \subset \mathbb
R\}.
\]
\end{definition}

\begin{definition} \rm
A number $\tau\in\mathbb R$ is said to be {\em $\varepsilon$-shift}
for $\varphi \in C(\mathbb R,X)$ if
$d(\varphi^{\tau},\varphi)<\varepsilon$.
\end{definition}

 Denote by
$\mathfrak{T}(\varepsilon,\varphi):=\{\tau\in\mathbb R:\
\rho(\varphi^{\tau},\varphi)<\varepsilon\}$ the set of all
$\varepsilon$-shifts of $\varphi$.

\begin{definition} \rm
A function $\varphi \in C(\mathbb R,X)$ is called {\em almost
recurrent (in the sense of Bebutov)} if for every $\varepsilon >0$
the set $\mathfrak{T}(\varphi,\varepsilon)$ is relatively dense.
\end{definition}

\begin{definition} \rm
A function $\varphi\in C(\mathbb R,X)$ is called {\em Lagrange
stable} if $\{\varphi^{h}:\ h\in \mathbb R\}$ is a precompact
subset of $C(\mathbb R,X)$.
\end{definition}

\begin{definition} \rm
A function $\varphi \in C(\mathbb R,X)$ is called {\em Birkhoff
recurrent} if it is almost recurrent and Lagrange stable.
\end{definition}

\begin{definition} \rm
A function $\varphi \in C(\mathbb R,X)$ is called:
\begin{enumerate}
\item[-] {\em Poisson stable} in the positive (respectively,
negative) direction if for every $\varepsilon >0$ and $l>0$ there
exists $\tau >l$ (respectively, $\tau <-l$) such that
$d(\varphi^{\tau},\varphi)<\varepsilon$. The function $\varphi$ is
called Poisson stable if it is Poisson stable in both directions;
\item[-] {\em strongly Poisson stable} if every function $p\in
H(\varphi)$ is Poisson stable.
\end{enumerate}
\end{definition}

In what follows, we denote as well $Y$ a complete metric space.

\begin{definition} \rm
%Let $\varphi \in C(\mathbb R,X)$ and $\psi \in C(\mathbb R,Y)$.
A function $\varphi\in C(\mathbb R,X)$ is called {\em Levitan almost
periodic} if there exists a Bohr almost periodic function $\psi \in
C(\mathbb R,Y)$ such that for any $\varepsilon >0$ there exists
$\delta =\delta (\varepsilon)>0$ such that $\mathcal
T(\psi,\delta)\subseteq \mathfrak T(\varphi,\varepsilon)$.
\end{definition}

\begin{remark} \rm
\begin{enumerate}
\item  Every Bohr almost periodic function is Levitan almost periodic.

\item  The function $\varphi \in C(\mathbb R,\mathbb R)$ defined by equality $\varphi(t)=\dfrac{1}{2+\cos t +\cos \sqrt{2}t}$
is Levitan almost periodic, but it is not Bohr almost periodic
\cite[ChIV]{Lev-Zhi}.
\end{enumerate}
\end{remark}

\begin{definition}\label{defAA1} \rm
A function $\varphi \in C(\mathbb R,X)$ is said to be {\em Bohr
almost automorphic} if it is Levitan almost periodic and Lagrange
stable.
\end{definition}

\begin{definition} \rm
A function $\varphi \in C(\mathbb R,X)$ is called {\em
quasi-periodic with the spectrum of frequencies
$\nu_1,\nu_2,\ldots,\nu_k$} if the following conditions are
fulfilled:
\begin{enumerate}
\item the numbers $\nu_1,\nu_2,\ldots,\nu_k$ are rationally independent;
\item there exists a continuous function $\Phi :\mathbb R^{k}\to X$ such that
$\Phi(t_1+2\pi,t_2+2\pi,\ldots,t_k+2\pi)=\Phi(t_1,t_2,\ldots,t_k)$
for all $(t_1,t_2,\ldots,t_k)\in \mathbb R^{k}$;
\item $\varphi(t)=\Phi(\nu_1 t,\nu_2 t,\ldots,\nu_k t)$ for $t\in \mathbb R$.
\end{enumerate}
\end{definition}

Let $\varphi \in C(\mathbb R,X)$. Denote by $\mathfrak N_{\varphi}$
(respectively, $\mathfrak M_{\varphi}$) the family of all sequences
$\{t_n\}\subset \mathbb R$ such that $\varphi^{t_n} \to \varphi$
(respectively, $\{\varphi^{t_n}\}$ converges) in $C(\mathbb R,X)$ as
$n\to \infty$.

By $\mathfrak N_{\varphi}^{u}$ (respectively, $\mathfrak
M_{\varphi}^{u}$) we denote the family of sequences $\{t_n\}\in
\mathfrak N_{\varphi}$ such that $\varphi^{t_n}$ converges to
$\varphi$ (respectively,  $\varphi^{t_n}$ converges) uniformly in
$t\in\mathbb R$ as $n\to \infty$.

\begin{remark}\rm
\begin{enumerate}
\item The function $\varphi \in C(\mathbb R,X)$ is pseudo-periodic in the positive
(respectively, negative) direction if and only if there is a
sequence $\{t_n\}\in \mathfrak N_{\varphi}^{u}$ such that $t_n\to
+\infty$ (respectively, $t_n\to -\infty$) as $n\to \infty$.

\item Let $\varphi \in C(\mathbb R,X)$, $\psi \in C(\mathbb R,Y)$ and
$\mathfrak N_{\psi}^{u} \subseteq \mathfrak N_{\varphi}^{u}$. If the
function $\psi$ is pseudo-periodic in the positive (respectively,
negative) direction, then so is $\varphi$.
\end{enumerate}
\end{remark}

\begin{definition}\label{defPR}\rm A function $\varphi \in C(\mathbb R,X)$ is called
\cite{Sch68,sib}  {\em pseudo-recurrent} if for any $\varepsilon
>0$ and $l\in\mathbb R$ there exists $L=L(\varepsilon,l)>0$ such
that for any $\tau_0\in \mathbb R$ we can find a number $\tau
=\tau(\varepsilon,l,t_0) \in [l,l+L]$ satisfying
$$
\sup\limits_{|t|\le 1/\varepsilon}\rho(\varphi(t+\tau_0
+\tau),\varphi(t+\tau_0))\le \varepsilon.
$$
\end{definition}

\begin{remark}\label{remPR} \rm (\cite{Sch68,sib})
\begin{enumerate}
\item Every Birkhoff recurrent function is pseudo-recurrent, but the inverse
statement is not true in general.

\item If the function $\varphi \in C(\mathbb R,X)$ is
pseudo-recurrent, then every function $\psi\in H(\varphi)$ is
pseudo-recurrent.

\item If  the function $\varphi \in C(\mathbb R,X)$  is
Lagrange stable and  every function $\psi\in H(\varphi)$ is Poisson
stable, then $\varphi$ is pseudo-recurrent.
\end{enumerate}
\end{remark}

\section{Linear systems}

\subsection{Linear non-autonomous dynamical systems with
exponential dichotomy}

Let $(E, |\cdot |)$ be a Banach space with the norm $|\cdot|$,
$\langle E, \varphi, (Y,\mathbb T, \sigma)\rangle$ (or shortly
$\varphi$) be a linear cocycle over dynamical system $(Y,\mathbb
T,\sigma)$ with the fiber $E$, i.e., $\varphi$ is a continuous
mapping from $\mathbb T\times E \times Y$ into $E$ satisfying the
following conditions:
\begin{enumerate}
\item $\varphi(0,u,y)=u$ for all $u\in E$ and $y\in Y$; \item
$\varphi(t+\tau,u,y)=\varphi(t,\varphi(\tau,u,y),\sigma(\tau,y))$
for all $t,\tau\in\mathbb T_{+}$, $u\in E$ and $y\in Y$; \item for
all $(t,y)\in \mathbb T_{+}\times Y$ the mapping
$\varphi(t,\cdot,y):E\mapsto E$ is linear.
\end{enumerate}

Denote by $[E]$ the Banach space of all linear bounded operators
$A$ acting on the space $E$ equipped with the operator norm
$||A||:=\sup\limits_{|x|\le 1}|Ax|$.

\begin{example}\label{exLS1}  Let $Y$ be a complete metric space,
$(Y,\mathbb R,\sigma)$ be a dynamical system on $Y$ and $ \Lambda
$ be some complete metric space of linear closed operators acting
into Banach space $ E $ ( for example $ \Lambda = \{ A_{0}+B | B
\in [E] \} $, where $ A_{0} $ is a closed operator that acts on $
E $). Consider the following linear differential equation
\begin{equation}\label{eqLS01}
x'=A(\sigma(t,y))x,\  \ (y\in Y)
\end{equation}
where $A\in C(Y,\Lambda)$. We assume that the following conditions
are fulfilled for equation (\ref{eqLS01}):
\begin{enumerate}
\item[a.] for any $ u \in E $ and $y\in Y $ equation
(\ref{eqLS01}) has exactly one solution that is defined on $
\mathbb R_{+} $ and satisfies the condition $ \varphi (0,u,y) = u
;$ \item[b.] the mapping $ \varphi : (t,u,y) \to \varphi (t,u,y) $
is continuous in the topology of $ \mathbb R_{+} \times E \times
Y$.
\end{enumerate}

Under the above assumptions the equation (\ref{eqLS01}) generates
a linear cocycle $\langle E, \varphi, (Y,$ $\mathbb R,$ $
\sigma)\rangle$ over dynamical system $(Y,\mathbb R,\sigma)$ with
the fiber $E$.
\end{example}

\begin{example}\label{exLS_02} Let $ \Lambda $ be
some complete metric space of linear closed operators acting into
Banach space $ E $. Consider the differential equation
\begin{equation}\label{eqLS02}
x'=A(t)x,
\end{equation}
where $A\in C(\mathbb R,\Lambda)$. Along this equation
(\ref{eqLS02}) consider its $H$-class, i.e., the following family
of equations
\begin{equation}\label{eqLS03}
x'=B(t)x,
\end{equation}
where $B\in H(A)$. We assume that the following conditions are
fulfilled for equation (\ref{eqLS02}) and its $H$-class
(\ref{eqLS03}):
\begin{enumerate}
\item[a.] for any $ u \in E $ and $ B \in H(A) $ equation
(\ref{eqLS03}) has exactly one mild solution $ \varphi (t,u,B)$
(i.e. $ \varphi (\cdot ,u,B ) $ is continuous, defined on $\mathbb
R_{+}$ and satisfies of equation
\begin{equation}\label{13.8.10*}
\varphi (t,v,B ) = U(t,B)v + \int_{0}
^{t}U(t-\tau,B^{\tau})\varphi (\tau,v,B )d\tau \nonumber
\end{equation}
and the condition $ \varphi (0, u, B ) = v $); \item[b.] the
mapping $\varphi : (t,u,B ) \to \varphi (t,u,B )$ is continuous in
the topology of $\mathbb R_{+} \times E \times C(\mathbb R;
\Lambda)$.
\end{enumerate}

Denote by $(H(A),\mathbb R,\sigma)$ the shift dynamical system on
$H(A)$. Under the above assumptions the equation (\ref{eqLS02})
generates a linear cocycle $\langle E, \varphi, (H(A),\mathbb R,
\sigma)\rangle$ over dynamical system $(H(A),\mathbb R,\sigma)$
with the fiber $E$.

Note that equation (\ref{eqLS02}) and its $H$-class can be written
in the form (\ref{eqLS01}). In fact. We put $Y:=H(A)$ and denote
by $\mathcal A \in C(Y,\Lambda)$ defined by equality $\mathcal
A(B):=B(0)$ for all $B\in H(A)=Y$, then $B(\tau)=\mathcal
A(\sigma(B,\tau)$ ($\sigma(\tau,B):=B^{\tau}$, where
$B^{\tau}(t):=B(t+\tau)$ for all $t\in\mathbb R$). Thus the
equation (\ref{eqLS02}) with its $H$-class can be rewrite as
follow
\begin{equation}\label{eqLS_04}
x'=\mathcal A(\sigma(t,B))x.  \ (B\in H(A))\nonumber
\end{equation}
\end{example}

We will consider example of partial differential equations which
satisfy the above conditions a.-b.

\begin{example}\label{exLS03}
{\em Consider the differential equation
\begin{equation}\label{eqLS05}
u'=(A_1 + A_2(t) )u,
\end{equation}
where $ A_1 $ is a sectorial operator that does not depend on $ t
\in \mathbb R $, and $ A_2 \in C(\mathbb R ,[E]) $. The results of
\cite{Hen}, \cite{Lev-Zhi} imply that equation (\ref{eqLS05})
satisfies conditions a.-b. from Example \ref{exLS02}.}
\end{example}

\begin{definition}\label{defED1}
Recall (see, for example, \cite[Ch.VI]{C-L}) that the linear
cocycle $\langle E, \varphi, (Y,\mathbb S, \sigma)\rangle$ is
hyperbolic (or equivalently, satisfies the condition of
exponential dichotomy), if there exists a continuous projection
valued function $P:Y \to [E]$ satisfying:
\begin{enumerate}
\item $P(\sigma(t,y))U(t,y)=U(t,y)P(y)$ for all $(t,y)\in \mathbb
T\times Y$: \item for all $(t,y)\in \mathbb T\times Y$ the
operator $U_{Q}(t,y)$ is invertible as an operator from $Im Q(y)$
to $Im Q(\sigma(t,y)),$ where $Q(y):=Id_{E}-P(y)$ and
$U_{Q}(t,y):=U(t,y)Q(y);$ \item there exist constants $\nu >0$ and
$\mathcal N>0$ such that
\begin{equation}\label{eqPQ1}
\Vert U_{P}(t,y)\Vert \le \mathcal N e^{-\nu t} \ \mbox{and}\
\Vert U_{Q}(t,y)^{-1}\Vert \le \mathcal N e^{-\nu t}
\end{equation}
for all $y\in Y$ and $t\in \mathbb S_{+},$ where
$U_{P}(t,y):=U(t,y)P(y)$ and $U(t,y)=\varphi(t,\cdot,y)$.
\end{enumerate}
\end{definition}

\begin{lemma}\label{lED_10} Suppose that the linear
cocycle $\langle E, \varphi, (Y,\mathbb S, \sigma)\rangle$ is
hyperbolic and $\gamma \in C(\mathbb R,E)$ is a full trajectory of
cocycle $\varphi$, i.e., there exists a point $y_0\in Y$ such that
$\gamma (t)=U(t-\tau,\sigma(\tau,y_0))\gamma(\tau)$ for any $t\ge
\tau$ and $\tau \in \mathbb R$. If
$$
\sup\{|\gamma(t)|;\ t\in
\mathbb R\}<+\infty\ ,
$$
then $\gamma(\tau)=0$ for any $\tau\in
\mathbb R$.
\end{lemma}
\begin{proof}
Note that
$$
\gamma(t)=U(t-\tau,\sigma(\tau,y_0))\gamma(\tau)=U(t-\tau,\sigma(\tau,y_0))P(\sigma(\tau,y_0))\gamma(\tau)+
$$
$$
U(t-\tau,\sigma(\tau,y_0))Q(\sigma(\tau,y_0))\gamma(\tau)
$$
for any $t\ge \tau$ and, consequently,
$C_1:=\sup\{|U(t-\tau,\sigma(\tau,y_0))Q(\sigma(\tau,y))\gamma(\tau)|:\
t\ge \tau\}<+\infty$. According to (\ref{eqPQ1}) we have
\begin{equation}\label{eqPQ2}
C_1\ge
|U(t-\tau,\sigma(\tau,y_0))Q(\sigma(\tau,y))\gamma(\tau)|\ge
\mathcal N e^{\nu (t-\tau)}|Q(\sigma(\tau,y_0))|\nonumber
\end{equation}
for any $t\ge \tau$ and, consequently,
$Q(\sigma(\tau,y_0))\gamma(\tau)=0$ for any $\tau \in \mathbb R$.
This means that $\gamma(\tau)=P(\sigma(\tau,y_0))\gamma(\tau)$ for
any $\tau\in\mathbb R$.

On the other hand $\gamma
(t)=U(t-\tau,\sigma(\tau,y_0))\gamma(\tau)=U(t-\tau,\sigma(\tau,y_0))Q(\sigma(\tau,y_0))\gamma(\tau)$
for any $t\ge \tau$. Taking into account (\ref{eqPQ1}) we obtain
\begin{equation}\label{eqPQ3}
|\gamma(t)|\le \mathcal N e^{-\nu (t-\tau)}C
\end{equation}
for any $t\ge \tau$ and $\tau\in \mathbb R$, where
$C:=\sup\{|\gamma(\tau)|:\ \tau\in\mathbb R\}$. Passing into limit
in (\ref{eqPQ3}) as $\tau \to -\infty$ we obtain $\gamma(t)=0$ for
any $t\in \mathbb R$. Lemma is proved.
\end{proof}

\subsection{Relationship between different definitions of
hyperbolicity}

Along with classical definition of hyperbolicity (Definition
\ref{defED1}) we will use an other definition given below. And in
this Subsection we establish the relation between two definitions
of hyperbolicity for linear homogeneous differential equations
with continuous (bounded) coefficients.

Let $A\in C(\mathbb R,\Lambda)$, $\Sigma_{A}:=\{A^{\tau}:\
A^{\tau}(t):=A(t+\tau)$\index{$\Sigma_{A}$} for all $t\in \mathbb
R\}$ and $H(A):=\overline{\Sigma}_{A}$\index{$H(A)$}, where by bar
is denoted the closure of the set $\Sigma_{A}$ in $C(\mathbb
R,\lambda)$. In this Subsection we will suppose that the
operator-function $A\in C(\mathbb R,\Lambda)$ is
regular\index{regular operator-function}, i.e., for all $B\in
H(A)$ there exists a unique solution $\varphi(t,u,B)$ of equation
\begin{equation}\label{eqUC02}
x'=B(t)x \ \ (B\in \Sigma_{A})
\end{equation}
with initial data $\varphi(0,u,B)=u$ defined on $\mathbb R_{+}$.

\begin{definition}\label{defED_c}  (Classical definition \cite{Cop_1978,DK_1970}) Let $A\in C(\mathbb R,{E})$
. Linear differential equation
\begin{equation}\label{eqA}
x'=A(t)x
\end{equation}
satisfies the exponential dichotomy if there exist a projection
$P(A):E\to E$ and the positive constants $\mathcal N$ and $\nu >0$
such that
\begin{eqnarray}\label{eqED_a}
& ||U(t,A)P(A)U^{-1}(\tau,A)||\le \mathcal N e^{-\nu (t-\tau)}\
\mbox{for any}\ t> \tau \nonumber \\ &\mbox{and}\\ \nonumber&
||U(t,A)(I-P(A))U^{-1}(\tau,A)||\le \mathcal N e^{\nu(t-\tau)}\
\mbox{for any}\ t<\tau .
\end{eqnarray}
\end{definition}

\begin{lemma}\label{lH}\cite[ChIII]{Che_2009} Suppose that equation (\ref{eqA})
satisfies the exponential dichotomy, $A\in C(\mathbb R,[E])$ and
$B\in H(A)$. Then equation (\ref{eqUC02}) also satisfies the
exponential dichotomy.
\end{lemma}

\begin{definition}\label{defUC2} Differential equation (\ref{eqA})
is said to be hyperbolic (satisfies the condition of exponential
dichotomy), if there are two projections: $P,Q:H(A)\mapsto [E]$
($P^2(B)=P(B)$ and $Q^2(B)=Q(B)$ for all $B\in H(A)$) such that
\begin{enumerate}
\item the mappings $P$ and $Q$ are continuous; \item
$P(B)+Q(B)=Id_{E}$ for all $B\in H(A)$; \item
$U(t,B)P(B)=P(B^{t})U(t,B)$ for all $t\in\mathbb R_{+}$ and $B\in
H(A)$, where $U(t,B):=\varphi(t,\cdot,B)$; \item the mapping
$U_{Q}(t,B):=U(t,B)Q(B): Im(Q(B)\mapsto Im(Q(B))$ is invertible;
\item there are positive numbers $\mathcal N$ and $\nu$ such that
$||U_{P}(t,B)||\le \mathcal Ne^{-\nu t}$ and
$||[U_{Q}(t,B)]^{-1}||\le \mathcal Ne^{-\nu t}$ for any $t\ge 0$,
where $U_{P}(t,B):=U(t,B)P(B)$ for any $t\in \mathbb R_{+}$.
\end{enumerate}
\end{definition}

\begin{remark}\label{remUC2} Note that the definition above of the
hyperbolicity means that the cocycle $\langle
E,\varphi,(H(A),\mathbb R,\sigma)\rangle$ generated by equation
(\ref{eqA}) is hyperbolic in the sense of Definition \ref{defED1}.
\end{remark}

\begin{lemma}\label{lDP01} Let $A\in C(\mathbb R,[E])$.
If (\ref{eqA}) is hyperbolic in the sense of Definition
\ref{defUC2}, then it is also so in the classical sense.
\end{lemma}
\begin{proof}
Let (\ref{eqA}) be hyperbolic in the sense of Definition
\ref{defUC2}, $P(A): H(A)\mapsto [E]$ projection-operator and
$\mathcal N,\nu$ positive constants which figure in definition of
hyperbolicity.

Denote by $P(s):=U(s,A)P(A)U^{-1}(\tau,A)$ and
$Q(s):=Id_{E}-P(s)$. It easy to check that $P^{2}(s)=P(s)$ for any
$s\in \mathbb R$. Since equation (\ref{eqA}) is hyperbolic (in the
sense of Definition \ref{defUC2}), then
\begin{equation}\label{eqDP13.1}
||U(t,A^{s})P(s)||\le \mathcal N e^{-\nu t}
\end{equation}
for any $(t,s)\in \mathbb R_{+}\times \mathbb R$ and
\begin{equation}\label{eqDP14.1}
||U(t,A^{s})Q(s)||\le \mathcal N e^{-\nu t}
\end{equation}
for any $(t,s)\in \mathbb R_{+}\times\mathbb R 0$, where
$Q(A^{s}):=Id_{E}-P(A^{s})$.

Note that
\begin{equation}\label{eqDP15.1}
U(t,A)P(A)U^{-1}(\tau,A)=U(t-\tau,A^{\tau})U(\tau,A)P(A)U^{-1}(\tau,A)=U(t-\tau,A^{\tau})P(s)
\end{equation}
for any $t>\tau$ and
\begin{equation}\label{eqDP16.1}
U(t,A)Q(A)U^{-1}(\tau,A)=U(t-\tau,A^{\tau})U(\tau,A)Q(A)U^{-1}(\tau,A)=U(t-\tau,A^{\tau})Q(s)
\end{equation}
for any $t<\tau$.

From (\ref{eqDP13.1})- (\ref{eqDP16.1}) it follows that
\begin{equation}\label{eqDP13}
||U(t,A)P(A)U^{-1}(\tau,A)||\le \mathcal N e^{-\nu (t-\tau)}
\end{equation}
and
\begin{equation}\label{eqDP14}
||U(t,A)Q(A)U^{-1}(\tau,A)||\le \mathcal N e^{-\nu (t-\tau)}
\end{equation}
for any $t>\tau$, because
$U(t,A)Q(A)U^{-1}(\tau,A)=(U_{Q}(t,\tau))^{-1},$ where
$U_{Q}(t,\tau)=U(t-\tau,A^{\tau})Q(A^{\tau})$. Lemma is proved.
\end{proof}

\begin{lemma}\label{lDP2} Let $E$ be a finite-dimensional Banach space.
%Suppose that the operator-function $A\in C(\mathbb R,[\mathfrak
%B])$is bounded on $\mathbb R$.
Then, if equation (\ref{eqA}) is hyperbolic (in the sense of
Definition \ref{defED_c}), then it is also hyperbolic in the sense
of Definition \ref{defUC2}.
\end{lemma}
\begin{proof} Let $A\in C(\mathbb R,[E])$. Denote by $\Sigma_{A}$ the set of all
translations of $A$, i.e., $\Sigma_{A}=\{A^{s}:\ s\in\mathbb R\}$
and $A^{s}(t):=A(t+s)$ for all $t\in\mathbb R$. Suppose that
equation (\ref{eqA}) is hyperbolic with projection $P(A)$ and the
constants $\mathcal N>0$ and $\nu >0$ which figure in Definition
\ref{defED_c}. Let $B\in \Sigma_{A}$, then there is a number $s\in
\mathbb R$ such that $B=A^{s}$. We put
$P(B):=U(s,A)P(A)U^{-1}(s,A),$ then it easy to check that $P(B)$
is a projection (i.e., $P^2(B)=P(B)$). Note that
\begin{equation}\label{eqDP13_1}
U(t,A^{s})P(A^{s})U(\tau,A^{s})=U(t+s,A)P(A)U^{-1}(\tau +s,A)
\end{equation}
and
\begin{equation}\label{eqDP14_1}
U(t,A^{s})Q(A^{s})U(\tau,A^{s})=U(t+s,A)Q(A)U^{-1}(\tau +s,A)
\end{equation}
for all $t,\tau,s\in \mathbb R$, where
$Q(A^{s}):=Id_{E}-P(A^{s})$. From (\ref{eqED_a}), (\ref{eqDP13_1})
and (\ref{eqDP14_1}) it follows that
\begin{equation}\label{eqDP15}
||U(t,B)P(B)U^{-1}(\tau,B)||\le \mathcal N e^{-\nu (t-\tau)} \
\mbox{for all}\ t>\tau
\end{equation}
and
\begin{equation}\label{eqDP16}
||U(t,B)Q(B)U^{-1}(\tau,B)||\le \mathcal N e^{-\nu (t-\tau)} \
\mbox{for all}\ t<\tau ,
\end{equation}
for all $B\in \Sigma_{A}$.
%Finally we note that by Corollary
%\ref{corDP1} the mappings $P,Q:\Sigma_{A}\mapsto [E]$
% uniformly continuous.
Denote by $H(A):=\overline{\{A^{s}:\ s\in\mathbb R\}}$, where by
bar is denoted the closure in the space $C(\mathbb R,[E])$. will
show that the mapping $P:\Sigma_{A}\mapsto [E]$ admits a unique
extension $P:H(A)\mapsto [E]$ possessing the following properties:
\begin{enumerate}
\item the mapping $P:H(A)\mapsto [E]$ is continuous; \item
$P(B^{\tau})=U(\tau,B)P(B)U^{-1}(\tau,B)$ for all $\tau\in \mathbb
R$ and $B\in H(A)$.
\end{enumerate}
Let $B\in H(A)$, then there exists a sequence $\{\tau_n\}\subset
\mathbb R$ such that $B=\lim\limits_{n\to \infty}A^{\tau_n}$.
Consider the sequence $\{P(A^{\tau_n}\}\subset [E]$ (respectively
$\{Q(A^{\tau_n}\}\subset [E]$). Under the conditions of Lemma
\ref{lDP2} the sequence $\{P(A^{\tau_n}\}$ (respectively
$\{Q(A^{\tau_n}\}$) is relatively compact in $[E]$. Now we will
establish that the sequence $\{P(A^{\tau_n}\}$ (respectively
$\{Q(A^{\tau_n}\}$ admits at most one limiting point. In fact, if
$P'$ (respectively $Q'$) is a limiting point of
$\{P(A^{\tau_n})\}$ then there exists a subsequence
$\{\tau_{n}'\}\subseteq \{\tau_n\}$ such that
$P'=\lim\limits_{n\to \infty}P(A^{\tau_n'})$,
$Q'=\lim\limits_{n\to \infty}Q(A^{\tau_n'})$, $P'^2=P'$
(respectively, $Q'^2=Q'$) and $P'+Q'=I_{E}$. By Lemma \ref{lH} the
equation
\begin{equation}\label{eqEDB}
y'=B(t)y
\end{equation}
is hyperbolic with the projections $P(B)$ and $Q(B)$ and theses
projections are defined uniquely. From the last fact we obtain
that $P'=P(B)$ (respectively, $Q'=Q(B)$, i.e., the sequence
$\{P(A^{\tau_n})\}$ (respectively, $P'=P(B)$ (respectively,
$Q'=Q(B)$, i.e., the sequence $\{P(A^{\tau_n})\}$) admits a unique
limiting point $P(B)$ (respectively, $Q(B)$). It is clear that the
mapping $P:\Sigma_{A}\mapsto [E]$ (respectively,
$Q:\Sigma_{A}\mapsto [E]$) admits a (unique) extension on $H(A)$
and it is defined by equality $P(B):=\lim\limits_{n\to
\infty}P(A^{\tau_n})$ (respectively, $P(B):=\lim\limits_{n\to
\infty}P(A^{\tau_n})$), where $\{\tau_n\}$ is a subsequence of
$\mathbb R$ such that $B=\lim\limits_{n\to \infty}A^{\tau_n})$.

Now we will prove that the mapping $P:H(A)\mapsto [E]$
(respectively, $Q:H(A)\mapsto [E]$) is continuous. Let $B\in H(A)$
and $\{B_k\}$ be a subset of $H(A)$ such that $d(B,B_{k})\le 1/k$
for all $k\in\mathbb N$. Since $B_k\in H(A)$, then there exists a
sequence $\{\tau_n^{k}\}$ such that $B_{k}=\lim\limits_{n\to
\infty}A_{\tau_n^{k}}$. Let $n_k\in\mathbb N$ such that
\begin{equation}\label{eqEDB1}
d(B_k,A^{\tau_{n}^{k}})\le 1/k\ \ \mbox{and}\ \
||P(B_k)-P(A^{\tau_{n}^{k}})||\le 1/k
\end{equation}
for all $n\ge n_{k}$. Consider the sequence $\{\tau_{k}'\}$, where
$\tau_{k}':=\tau_{n_k}^{k}$, and note that from (\ref{eqEDB1}) we
have $\lim\limits_{k\to \infty}A^{\tau_{k}'}=B$ and, consequently
(see above)
\begin{equation}\label{eqEDB2}
\lim\limits_{k\to \infty}P(A^{\tau_{k}'})=P(B).
\end{equation}
Thus we have
\begin{equation}\label{eqEDB3}
||P(B_k)-P(B)||\le
||B_k-P(A^{\tau_{k}'})||+||P(A^{\tau_{k}'})-P(B)||
\end{equation}
for all $k\in \mathbb N$. Passing to limit in (\ref{eqEDB3}) as
$k\to \infty$ and taking in consideration (\ref{eqEDB1}) and
(\ref{eqEDB2}) we obtain $P(B)=\lim\limits_{k\to \infty}P(B_k)$,
i.e., the mapping $P:H(A)\mapsto [E]$ is continuous. Analogously
may be established the continuity of the mapping $Q:H(A)\mapsto
[E]$.

Let $B\in H(A)$ and $\tau \in \mathbb R$, then there exists a
sequence $\{\tau_n\}\subset \mathbb R$ such that $A^{\tau_n}\to B$
as $n\to \infty$ and, consequently, $U(t,A^{\tau_n})\to U(t,B)$,
$U^{-1}(\tau,B)\to U^{-1}(\tau,B)$ and $P(A^{\tau_n})\to P(B)$ as
$n\to \infty$. From the above facts we obtain
\begin{eqnarray}\label{eqEDB4}
&P(B^{\tau})=P(\lim\limits_{n\to \infty}A^{\tau
+\tau_{n}})=\lim\limits_{n\to \infty}P(A^{\tau +\tau_{n}})=
 \\
&\lim\limits_{n\to
\infty}U(\tau,A^{\tau_{n}})P(A^{\tau_n})U^{-1}(\tau,A^{\tau_{n}})=U(\tau,B)P(B)U^{-1}(\tau,B).\nonumber
\end{eqnarray}
Analogously we can prove that
$Q(B^{\tau})=U(\tau,B)Q(B)U^{-1}(\tau,B)$. Thus equation
(\ref{eqA}) is hyperbolic in the sense of Definition \ref{defUC2}.
Lemma is proved.
\end{proof}

\begin{lemma}\label{l13.2.1}\cite[Ch.III]{Che_2012}
Let $A\in C(\mathbb R,[E])$ and $\varphi(t,u,A)$ be a unique
solution of equation (\ref{eqA}) with initial data
$\varphi(0,u,A)=u$. Then the following statements hold:
\begin{enumerate}
\item[(i)] the map $A\mapsto U ( \cdot , A)$ of $C(\mathbb R,
[E])$ to $C(\mathbb R, [E])$ is continuous, where $U (t, A)$ is
the Cauchy's operator \cite{DK_1970} of equation (\ref{eqA}) and
\item[(ii)] the map $(t, u, A)\mapsto \varphi(t, u, A)$ of
$\mathbb R\times E \times C(\mathbb R, [E])$ to $E$ is continuous.
\end{enumerate}
\end{lemma}

\begin{theorem}\label{thDP1} Let $A\in C(\mathbb R,[E])$ and $E$ be finite-dimensional. Then equation (\ref{eqA}) is
hyperbolic in the sense of Definition \ref{defUC2} if and only if
it is hyperbolic in the sense of Definition \ref{defED_c}.
\end{theorem}
\begin{proof} This statement follows from Lemmas \ref{lDP01} and
\ref{lDP2}.
\end{proof}

\begin{remark}\label{remDP} Theorem \ref{thDP1} (respectively, Lemma \ref{lDP2}) remains true also
for the infinite-dimensional equations under the following
condition: $A\in C(\mathbb R,[E])$ and the family of projections
$\{P(t)\}:=\{U(t,A)P(A)U^{-1}(t,A):\ t\in\mathbb R\}\subset [E]$
is precompact in $[E]$.
\end{remark}

\begin{theorem}\label{thDP2} Suppose that the following conditions
are fulfilled:
\begin{enumerate}
\item the operator-function $A\in C(\mathbb R,[E])$ is
$\tau$-periodic (respectively, quasi-pe\-ri\-o\-dic, almost
periodic, almost automorphic, recurrent); \item equation
(\ref{eqA}) is hyperbolic (satisfies the condition of exponential
di\-cho\-to\-my).
\end{enumerate}

Then the operator-function the operator-functions
$P(t):=U(t,A)P(A)U^{-1}(t,A)$ and $Q(t):=U(t,A)Q(A)U^{-1}(t,A)$
are $\tau$-periodic (respectively, quasi-pe\-ri\-o\-dic, almost
periodic, almost automorphic, recurrent) and $\mathfrak
M_{A}\subseteq \mathfrak M_{P}$.
\end{theorem}
\begin{proof} We will prove this statement for the
operator-function $P(t)$, because the statement for the
operator-function $Q(t)$ can be proved using the same arguments.

Let $A\in C(\mathbb R,[E])$ be  $\tau$-periodic (respectively,
quasi-periodic, almost periodic, almost automorphic, recurrent),
then the operator-function $A$ is bounded on $\mathbb R$ and by
Theorem \ref{thDP1} (see its proof) the function
$\tilde{h}:\Sigma_{A}\mapsto [E]$, defined by equality
$\tilde{h}(A^{s})=P(s)$ ($\forall \ s\in\mathbb R$), is uniformly
continuous function and, consequently, it admits a unique
continuous extension on $\overline{\Sigma}_{A}$. Consider the
mapping $h:H(A)\mapsto H(P)$ defined by $h(B)=P_{B}$ ($\forall \
B\in H(A)$), where $P_{B}:\mathbb R\mapsto [E]$ and
\begin{equation}\label{eqPB0}
P_{B}(t):=U(t,B)P(B)U^{-1}(t,B) \ (\forall \ t\in\mathbb R).
\end{equation}

Note that the map $h$ possesses the following properties:
\begin{enumerate}
\item $h(A)=P_{A}$; \item $h(B^{s})=P_{B}^{s}$ for any
$s\in\mathbb R$, where $P^{s}_{B}$ is the $s$-translation of
$P_{B}$, i.e.,
\begin{equation}\label{eqPB1}
P^{s}_{B}(t)=P_{B}(t+s)
\end{equation}
for any $t\in\mathbb R$; \item $h$ is continuous.
\end{enumerate}
In fact. The first statement is evident. To establish the second
statement we note that
\begin{eqnarray}\label{eqPB2}
& P_{B^{s}}(t)=U(t,B^s)P(B^s)U^{-1}(t,B^s)=\nonumber \\
& U(t,B^s)U(s,B)P(B)U^{-1}(s,B)U^{-1}(t,B^s)=\nonumber
\\
& U(t+s,B)P(B)U^{-1}(t+s,B)=P_{B}(t+s)
\end{eqnarray}
for any $t,s\in\mathbb R$. The second statement follows from
(\ref{eqPB0})-(\ref{eqPB2}).

Let now $B\in H(A)$ be an arbitrary point and $\{B_n\}\subset
H(A)$ such that $B_n\to B$ as $n\to \infty$ (in the topology of
the space $C(\mathbb R,[E])$), then $h(B_n)=P_{B_n}$. Since
$P_{B_n}(t)=U(t,B_n)P(B_n)U^{1}(t,B_n)$, then $h(B_n)\to h(B)$ as
$n\to \infty$, because $P(B_n)\to P(B)$ (see the proof of the
Theorem \ref{thDP1}) and conform Lemma \ref{l13.2.1} we have
$U(t,B_n)\to U(t,B)$ (respectively, $U^{-1}(t,B_n)\to
U^{-1}(t,B)$) in the topology of the space $C(\mathbb R,[E])$.

Under the conditions of Theorem the point $A$ (of the shift
dynamical system $(C(\mathbb R,[E]),\mathbb R,\sigma$)) is stable
in the sense of Lagrange and the point $P_{A}$ (of the shift
dynamical system $(C(\mathbb R,[E]),\mathbb R,\sigma$) is
uniformly comparable with $A$ by character of recurrence. Now to
finish the proof of Theorem it is sufficient to apply Remark
\ref{r4.1*}.
\end{proof}

\begin{coro}\label{corBP1} Under the condition of Theorem \ref{thDP2},
if $A$ is almost periodic (respectively, almost automorphic), then
the operator-function $P$ is also almost periodic (respectively,
almost automorphic) and its frequency module is contained in the
frequency module of $A$.
\end{coro}

\begin{remark} Note that for finite-dimensional equations ($\dim E
<\infty$), if
\begin{enumerate}
\item[-] the operator-function $A\in C(\mathbb R.[E])$ is almost
periodic, then Corollary \ref{corBP1} is well known (see
\cite{Cop_1967} and \cite{Cop_1978}); \item[-] the
operator-function $A\in C(\mathbb R.[E])$ is almost automorphic,
then Corollary \ref{corBP1} was proved in the work
\cite{Lin_1987}.
\end{enumerate}
\end{remark}

\subsection{Linear non-homogeneous (affine) dynamical systems}
Let $\langle E, \varphi, (Y,\mathbb S,$ $\sigma)\rangle$ be a
linear cocycle over dynamical system $(Y,\mathbb R,\sigma)$ with
the fiber $E$, $f\in C(Y,\mathbb B)$ and $\psi$ be a mapping from
$\mathbb T\times E \times Y$ into $E$ defined by equality
\begin{equation}\label{eqLS5}
\psi(t,u,y):=U(t,y)u+\int_{0}^{t}U(t-\tau,\sigma(\tau,y))f(\sigma(\tau,y))d\tau
\ \ \mbox{if}\ \mathbb S=\mathbb R
\end{equation}
and
\begin{equation}\label{eqLS6}
\psi(t,u,y):=U(t,y)u+ \sum_{\tau
=0}^{t}U(t-\tau,\sigma(\tau,y))f(\sigma(\tau,y)) \ \ \mbox{if}\
\mathbb S=\mathbb Z.
\end{equation}

From the definition of the mapping $\psi$ it follows that $\psi$
possesses the following properties:
\begin{enumerate}
\item[1.] $\psi(0,u,y)=u$ for any $(u,y)\in E\times Y$; \item[2.]
$\psi(t+\tau,u,y)=\psi(t,\psi(\tau,u,y),\sigma(\tau,y))$ for any
$t,\tau\in \mathbb T$ and $(u,y)\in E\times Y$; \item[3.] the
mapping $\psi :\mathbb T\times E\times Y\mapsto E$ is continuous;
\item[4.] $\psi (t,\lambda u+\mu v,y)=\lambda \psi(t,u,y)+\mu
\psi(t,v,y)$ for any $t\in\mathbb T$, $u,v\in E$, $y\in Y$ and
$\lambda,\mu\in \mathbb R$ (or $\mathbb C$) with condition
$\lambda +\mu =1$, i.e., the mapping $\psi(t,\cdot,y):E\mapsto E$
is affine for every $(t,y)\in \mathbb T\times Y$.
\end{enumerate}

\begin{definition}\label{defAF1} A triplet $\langle E,\varphi, (Y,\mathbb
S,Y)\rangle$is called an affine (non-homogeneous) cocycle
\index{an affine (non-homogeneous) cocycle} over dynamical system
$(Y,\mathbb T,Y)$ with the fiber $E$, if the $\varphi$ is a
mapping from $\mathbb T\times E\times Y$ into $E$ possessing the
properties 1.-4.
\end{definition}

\begin{remark}\label{remNH1} If we have a linear cocycle $\langle E, \varphi,
(Y,\mathbb R, \sigma)\rangle$ over dynamical system $(Y,\mathbb
R,\sigma)$ with the fiber $E$ and $f\in C(Y,\mathbb B)$, then by
equality (\ref{eqLS5}) (respectively, by (\ref{eqLS6})) is defined
an affine cocycle $\langle E, \psi, (Y,\mathbb R, \sigma)\rangle$
over dynamical system $(Y,\mathbb S,\sigma)$ with the fiber $E$
which is called an affine (non-homogeneous) cocycle associated by
linear cocycle $\varphi$ and the function $f\in C(Y,E)$.
\end{remark}

\begin{example}\label{exNH1}  Let $Y$ be a complete metric space,
$(Y,\mathbb R,\sigma)$ be a dynamical system on $Y$ and $ \Lambda
$ be some complete metric space of linear closed operators acting
into Banach space $ E $ and $f\in C(Y,E)$. Consider the following
linear non-homogeneous differential equation
\begin{equation}\label{eqNH1}
x'=A(\sigma(t,y))x +f(\sigma(t,y)),\  \ (y\in Y)
\end{equation}
where $A\in C(Y,\Lambda)$. We assume that conditions a. and b.
from Example \ref{exLS1} are fulfilled for equation
(\ref{eqLS01}).

Under the above assumptions equation (\ref{eqLS01}) generates a
linear cocycle $\langle E, \varphi, (Y,\mathbb R, \sigma)\rangle$
over dynamical system $(Y,\mathbb R,\sigma)$ with the fiber $E$.
According to Remark \ref{remNH1} by equality (\ref{eqLS5}) is
defined a linear non-homogeneous cocycle $\langle
E,\psi,(Y,\mathbb R,\sigma)\rangle$ over dynamical system
$(Y,\mathbb R,\sigma)$ with the fiber $E$. Thus every
non-homogeneous linear differential equations (\ref{eqNH1}), under
conditions a. and b. generates a linear non-homogeneous cocycle
$\psi$.
\end{example}

\begin{example}\label{exLS02}{\rm Let $ \Lambda $ be
some complete metric space of linear closed operators acting into
Banach space $ E $ and $f\in C(\mathbb R,E)$. Consider a linear
non-homogeneous differential equation
\begin{equation}\label{eqNH2}
x'=A(t)x +f(t),
\end{equation}
where $A\in C(\mathbb R,\Lambda)$. Along this equation
(\ref{eqNH2}) consider its $H$-class, i.e., the following family
of equations
\begin{equation}\label{eqNH3}
x'=B(t)x +g(t),
\end{equation}
where $(B,g)\in H(A,f)$. We assume that the following conditions
are fulfilled for equation (\ref{eqNH2}) and its $H$-class
(\ref{eqNH3}):
\begin{enumerate}
\item[a.] for any $ u \in E $ and $ B \in H(A) $ equation
(\ref{eqNH3}) has exactly one mild solution $ \varphi (t,u,B)$ and
the condition $ \varphi (0, u, B ) = v ;$ \item[b.] the mapping $
\varphi : (t,u,B ) \to \varphi (t,u,B ) $ is continuous in the
topology of $\mathbb R_{+} \times E \times C(\mathbb R ; \Lambda
)$.
\end{enumerate}

Denote by $(H(A,f),\mathbb R,\sigma)$ the shift dynamical system
on $H(A,f)$. Under the above assumptions the equation
(\ref{eqNH2}) generates a linear cocycle $\langle E, \varphi,
(H(A,f),\mathbb R, \sigma)\rangle$ over dynamical system
$(H(A,f),\mathbb R,\sigma)$ with the fiber $E$. Denote by $\psi$ a
mapping from $\mathbb R_{+}\times E\times H(A,f)$ into $E$ defined
by equality
\begin{equation}\label{eqNH4}
\psi(t,u,(B,g)):=U(t,B)u+\int_{0}^{t}U(t-\tau,B^{\tau})g(\tau)d\tau,\nonumber
\end{equation}
then $\psi$ possesses the following properties:
\begin{enumerate}
\item[(i)] $\psi(0,u,(B,g))=u$ for any $(u,(B,g))\in E\times
H(A,f)$; \item[(ii)]
$\psi(t+\tau,u,(B,g))=\psi(t,\psi(\tau,u,(B,g)),(B^{\tau},g^{\tau}))$
for any $t,\tau\in \mathbb T$ and $(u,(B,g))\in E\times H(A,f)$;
\item[(iii)] the mapping $\psi :\mathbb T\times E\times
H(A,f)\mapsto E$ is continuous; \item[(iv)] $\psi (t,\lambda u+\mu
v,(B,g))=\lambda \psi(t,u,(B,g))+\mu \psi(t,v,(B,g))$ for any
$t\in\mathbb T$, $u,v\in E$, $(B,g)\in H(A,f)$ and $\lambda,\mu\in
\mathbb R$ (or $\mathbb C$) with condition $\lambda +\mu =1$,
i.e., the mapping $\psi(t,\cdot,(B,g)):E\mapsto E$ is affine for
every $(t,(B,g))\in \mathbb T\times H(A,f)$.
\end{enumerate}

Thus, every linear non-homogeneous differential equation of the
form (\ref{eqNH2}) (and its $H$-class (\ref{eqNH3})) generates a
linear non-homogeneous cocycle $\langle E, \psi, (H(A,f),\mathbb
R,\sigma)\rangle$ over dynamical system $(H(A,f),\mathbb
R,\sigma)$ with the fiber $E$.}
\end{example}

\begin{remark}\label{remED1}  1. If $\Lambda =[E]$, $A\in C(\mathbb R,[E])$ and
, then according to Lemma \ref{l13.2.1}  conditions a. and b. in
Example are fulfilled. Thus equation (\ref{eqNH2}) with operator
function $A\in C(\mathbb R,[E])$ generates a linear
non-homogeneous cocycle $\psi$.

2. A closed linear operator $A: D(A)\to E$ with dense domain
$D(A)$ is said \cite{Hen} to be \emph{sectorial}\index{sectorial
operator} if one can find a $\phi\in (0, \frac{\pi}{2}),$ an $M\ge
1,$ and a real number $a$ such that the sector $$ S_{a,\phi} :=
\{\lambda \ |\ |\mbox{arg}\ (\lambda-a)| \le\pi , \lambda \not=a\}
$$ lies in the resolvent set $\rho(A)$ of $A$ and $\Vert (\lambda
I -A)^{-1}\Vert \le M |\lambda - a|^{-1}$ for any $\lambda \in
S_{a,\phi}.$ An important class of sectorial operators is formed
by elliptic operators \cite{Hen}, \cite{Ize}.

Consider the differential equation
\begin{equation}\label{eq13.2.14}
u' = (A_1 + A_2(t))u,
\end{equation}
where $A_1$ is a sectorial operator that does not depend on $t\in
\mathbb R,$ and $A_2\in C(\mathbb R, [E]).$

The results of \cite{Hen,Lev-Zhi},  imply that equation
(\ref{eq13.2.14}) satisfies conditions (i)-(iii).
\end{remark}

Note that equation (\ref{eqNH2}) (and its $H$-class (\ref{eqNH3}))
can be written in the form (\ref{eqLS01}). In fact. We put
$Y:=H(A,g)$ and denote by $\mathcal A \in C(H(A,f),\Lambda)$
(respectively, $\mathcal f\in C(H(A,f),E)$) defined by equality
$\mathcal A(B,g):=B(0)$ (respectively, $\mathcal f(B,g)=g(0)$) for
any $(B,g)\in H(A,f)$, then $B(\tau)=\mathcal
A(B^{\tau},g^{\tau})$ (respectively, $g(\tau)=\mathcal
f(B^{\tau},g^{\tau})$), where $B^{\tau}(t):=B(t+\tau)$ and
$g^{\tau}(t):=g(t+\tau)$ for any $t\in\mathbb R$). Thus the
equation (\ref{eqNH2}) with its $H$-class can be rewrite as follow
\begin{equation}\label{eqLS04}
x'=\mathcal A(\sigma(t,B))x +\mathcal F(\sigma(t,B)).  \ (B,g)\in
H(A,f))\nonumber
\end{equation}

\section{Linear Stochastic Differential Equations}

Consider the linear nonhomogeneous equation
\begin{equation}\label{eqLN1}
\dot{x}=A(t)x+f(t)\nonumber
\end{equation}
on the Banach space $E$, where $f\in C(\mathbb R,E)$ and $A(t)$
generates a cocycle $\langle E,$ $\varphi,$ $(H(A,f),$ $\mathbb
R,$ $\sigma)\rangle$ (or shortly $\varphi$) with fiber $E$ and
base (driving system) $(H(A,f),$ $\mathbb R,$ $\sigma)$.

Denote by $C_{b}(\mathbb R,E)$ the Banach space of all continuous
and bounded mappings $\varphi :\mathbb R\to E$ equipped with the
norm $||\varphi||_{\infty}:=\sup\{|\varphi(t)|:\ t\in\mathbb R\}$.

Let $(H,|\cdot|)$ be a real separable Hilbert space, $(\Omega ,
\mathcal F,\mathbb P)$ be a probability space, and $L^{2}(\mathbb
P,H)$ be the space of $H$-valued random variables $x$ such that
\begin{equation*}\label{eqLSDE03}
\mathbb E|x|^2 :=\int\limits_{\Omega}|x|^2 d\mathbb P<\infty .
\end{equation*}
Then $L^2(\mathbb P,H)$ is a Hilbert space equipped with the norm
\begin{equation*}\label{eqLSDE04}
||x||_2:=\Big{(}\int\limits_{\Omega}|x|^2  d\mathbb P\Big{)}^{1/2}.
\end{equation*}
For $f\in C_b(\mathbb R, L^2(\mathbb P,H))$, the space of bounded
continuous mappings from $\mathbb R$ to $L^2(\mathbb P,H)$, we
denote $||f||_{\infty}:=\sup\limits_{t\in\mathbb R}||f(t)||_2$.

Consider the following linear stochastic differential equation
\begin{equation}\label{eqSDE}
dx(t)=(A(t)x(t)+f(t)dt+g(t)dW(t), %\ \ t\in \mathbb R,
\end{equation}
where $f,g\in C(\mathbb R,H)$, $A(t)$ is an generator of cocycle
$\varphi$ and $W(t)$ is a two-sided standard one-dimensional
Brownian motion defined on the probability space $(\Omega,\mathcal
F,\mathbb P)$. We set $\mathcal F_{t}:=\sigma\{W(u):  u \le t\}$.

Recall that an $\mathcal F_{t}$-adapted processes
$\{x(t)\}_{t\in\mathbb R}$ is said to be a mild solution of
equation (\ref{eqSDE}) defined on interval $I=(a,b)\subset \mathbb
R$ if it satisfies the stochastic integral equation
$$
x(t)=U(t-s,A^{s})x(s)+\int_{s}^{t}U(t-\tau,A^{\tau})f(\tau)ds+\int_{s}^{t}U(t-\tau,A^{\tau})g(\tau)dW(\tau)
,
$$
for any $t\ge s$ and each $t,s\in I$ ($t\ge s$).

Let $\mathcal P(H)$ be the space of all Borel probability measures
on $H$ endowed with the $\beta$ metric:
$$
\beta (\mu,\nu) :=\sup\left\{ \left| \int f d \mu - \int fd
\nu\right|: ||f||_{BL} \le 1 \right\}, \quad \hbox{for }\mu,\nu\in
\mathcal P(H),
$$
where $f$ are bounded Lipschitz continuous real-valued functions on
$H$ with the norms
\[
||f||_{BL}= Lip(f) + ||f||_\infty,~ Lip(f)=\sup_{x\neq y}
\frac{|f(x)-f(y)|}{|x-y|},~ ||f||_{\infty}=\sup_{x\in H}|f(x)|.
\]
Recall that a sequence $\{\mu_n\}\subset \mathcal P(H)$ is said to
{\em weakly converge} to $\mu$ if $\int f d\mu_n\to \int f d\mu$
for any $f\in C_b(H)$, where $C_b(H)$ is the space of all bounded
continuous real-valued functions on $H$. It is well-known that
$(\mathcal P(H),\beta)$ is a separable complete metric space and
that a sequence $\{\mu_n\}$ weakly converges to $\mu$ if and only
if $\beta(\mu_n, \mu) \to0$ as $n\to\infty$.

\begin{definition}\label{def6.1} A sequence of random variables $\{x_n\}$ is
said to \emph{converge in distribution} to the random variable $x$
if the corresponding laws $\{\mu_n\}$ of $\{x_n\}$ weakly converge
to the law $\mu$ of $x$.
\end{definition}

\begin{example}\label{exLSDE}{\rm Let $ \Lambda $ be
some complete metric space of linear closed operators acting into
Hilbert space $ H $ and $f,g\in C(\mathbb R,H)$. Consider a linear
non-homogeneous stochastic differential equation (\ref{eqSDE}),
where $A\in C(\mathbb R,\Lambda)$. Along this equation
(\ref{eqSDE}) consider its $H$-class, i.e., the following family
of equations
\begin{equation}\label{eqSDE_B}
dx(t)=(\tilde{A}(t)x(t)+\tilde{f}(t))dt+\tilde{g}(t)dW(t), %\ \ t\in \mathbb R,
\end{equation}
where $(\tilde{A},\tilde{f},\tilde{g})\in
H(A,f,g):=\overline{\{(A^{\tau},f^{\tau}.g^{\tau}):\ \tau\in
\mathbb R\}}$ and by bar is denoted the closer in product space
$C(\mathbb R,\Lambda)\times C(\mathbb R,H)\times C(\mathbb R,H)$.
We assume that the following conditions are fulfilled for equation
\begin{equation}\label{eqSDE_A}
dx(t)=A(t)x(t)dt, %\ \ t\in \mathbb R,
\end{equation}
and its $H$-class
\begin{equation}\label{eqSDE_B2}
dx(t)=\tilde{A}(t)x(t)dt, %\ \ t\in \mathbb R,
\end{equation}
where $\tilde{A}\in H(A):=\overline{\{A^{\tau}:\ \tau\in \mathbb
R\}}$ and by bar is denoted the closer in $C(\mathbb R,H)$:
\begin{enumerate}
\item[a.] for any $ u \in H $ and $ \tilde{A} \in H(A) $ equation
(\ref{eqSDE_B}) has exactly one mild solution $ \varphi
(t,u,\tilde{A})$ with the condition $ \varphi (0, u, \tilde{A} ) =
u ;$ \item[b.] the mapping $ \varphi : (t,u,\tilde{A}) \to \varphi
(t,u,\dot{\tilde{A}} ) $ is continuous in the topology of $\mathbb
R_{+} \times H \times C(\mathbb R ; \Lambda )$.
\end{enumerate}

Denote by $(H(A,f,g),\mathbb R,\sigma)$ the shift dynamical system
on $H(A,f,g)$. Under the above assumptions the equation
(\ref{eqSDE_A}) generates a linear cocycle $\langle E, \varphi,
(H(A),\mathbb R, \sigma)\rangle$ over dynamical system
$(H(A),\mathbb R,\sigma)$ with the fiber $H$.

Define the mapping
\begin{equation}\label{eqP0}
\Phi: \mathbb R_{+} \times \mathcal P(H) \times H(A,f,g)  \to
\mathcal P(H),\nonumber
\end{equation}
with $\Phi(t,\mu,(\tilde{A},\tilde f,\tilde g))$ being the law (or
distribution) $\mathcal
L(\varphi(t,x,\tilde{A},\tilde{f},\tilde{g}))$ of the solution $
\varphi(t,x,\tilde{A},\tilde{f},\tilde{g})$ of equation
\begin{equation}\label{tilfg}
dx =(\tilde{A(t)}x+ \tilde f(t)) dt + \tilde g(t) d W, \quad
x(0)=x,
\end{equation}
where $\mathcal L(x)=\mu$ and $\Phi(0,\mu,(\tilde{A},\tilde
f,\tilde g))=\mu$.

We have the following result on $\Phi$:

\begin{theorem}\label{cocycle}\cite{CL_2016},\cite{CL_2016.1}
The mapping $\Phi$ is a continuous cocycle with base (driving
system) $(H(A,f,g),\mathbb R,\sigma)$ and fiber $\mathcal P(H)$,
i.e., the mapping $\Phi: \mathbb R_{+} \times \mathcal P(H) \times
H(A,f,g) \to \mathcal P(H)$ is continuous and satisfies
\begin{equation}\label{cocy}
\Phi(0,\mu, (\tilde{A},\tilde f,\tilde g))=\mu,\quad \Phi(t+\tau,
\mu, (\tilde{A},\tilde f,\tilde g))= \Phi(t,\Phi(\tau,\mu,
(\tilde{A},\tilde f,\tilde g)), (\tilde{A}^{\tau},\tilde
f^\tau,\tilde g^\tau))
\end{equation}
for any $t,\tau\ge 0$, $(\tilde{A},\tilde f, \tilde g)\in
H(A,f,g)$ and $\mu\in\mathcal P(H)$.
\end{theorem}

\begin{proof}
(1) The mapping $\Phi$ satisfies the cocycle property
\eqref{cocy}. Note that $\Phi(0,\mu,$ $ (\tilde{A},\tilde f,\tilde
g))=\mu$ by its definition. Under our assumptions \eqref{tilfg}
admits a unique solution which we denote by
$\varphi(t,x,(\tilde{A},\tilde f,\tilde g, W))$ with
$\varphi(0,x,(\tilde{A},\tilde f,\tilde g,W))=x$. Let
$\tilde{W}^\tau(t)= W(t+\tau) - W(\tau)$. Then $\tilde{W}^\tau$ is
still a Brownian motion which shares the same distribution as that
of $W$.

Define
\begin{equation}\label{eqP1}
\phi(t) := \varphi(t,x,(\tilde{A},\tilde f,\tilde g,W)), \quad
\psi(t):=\varphi(t,\phi(\tau),(\tilde{A}^{\tau},\tilde
f^\tau,\tilde g^\tau, \tilde{W}^\tau)), \quad
\eta(t):=\phi(t+\tau).
\end{equation}
Then $\phi(t)$ is a solution of \eqref{tilfg} with $\phi(0)=x$,
and $\psi(t)$ is a solution of
\begin{equation}\label{tauE}
dx= (\tilde{A}^{\tau}(t)x+\tilde f^\tau(t))dt + \tilde g^\tau (t)
d \tilde{W}^\tau
\end{equation}
with $\psi(0)=\phi(\tau)=\varphi(\tau,x,(\tilde{A},\tilde f,\tilde
g,W))$. On the other hand, note that $\eta(t)$ is also a solution
of \eqref{tauE} with $\eta(0)=\phi(\tau)$. So by the uniqueness of
solutions of \eqref{tauE}, we get $\eta(t)=\psi(t)$. That is
\begin{equation}\label{equ}
\varphi(t+\tau, x, (\tilde{A},\tilde f,\tilde g,W)) =
\varphi(t,\varphi(\tau, x, (\tilde{A},\tilde f,\tilde g, W)),
(\tilde{A}^{\tau},\tilde f^\tau, \tilde g^\tau, \tilde{W}^\tau)).
\end{equation}

Thus for any Brownian motion $\bar W$ and any random variable
$\bar x$ which has the same distribution as that of $x$, the
solution of the equation
\begin{equation}\label{eqP2}
dx= (\tilde{A}(t)x+\tilde f(t))dt + \tilde g(t) d \bar W, \quad
x(0)=\bar x \nonumber
\end{equation}
admits the same law on $H$ as $\phi(t)$ above. In other words, the
law of solution is uniquely determined by the coefficients
$\tilde{A},\tilde f,\tilde g$ and the initial distribution $\mu$.
So
%when we consider the laws of solutions, we need not indicate explicitly the
%Brownian motion and the initial value. In this way,
we can simply denote the law of $\varphi(t,x,(\tilde{A},\tilde
f,\tilde g,W))$ by $\Phi(t,\mu, (\tilde{A},\tilde f,\tilde g))$.
Therefore, it follows from \eqref{equ} that the cocycle property
holds for $\Phi$:
\begin{equation}\label{eqP3}
\Phi(t+\tau, \mu, (\tilde{A},\tilde f,\tilde g))=
\Phi(t,\Phi(\tau,\mu, (\tilde{A},\tilde f,\tilde
g)),(\tilde{A}^{\tau},\tilde f^\tau,\tilde g^\tau)).\nonumber
\end{equation}

To prove the continuity of $\Phi$ it is sufficient to apply
Proposal 3.1 (item c.) from \cite{DT_1995}.
\end{proof}
}
\end{example}

\begin{definition}\label{def6} Let $\{\varphi (t)\}_{t\in\mathbb R}$ be a
mild solution of equation \eqref{eqSDE}. Then $\varphi$ is called
$\tau$-periodic (respectively, quasi periodic, Bohr almost
periodic, almost automorphic, recurrent, Levitan almost periodic,
almost recurrent, Poisson stable) {\em in distribution} if the
function $\phi \in C(\mathbb R,\mathcal P(H))$ is $\tau$-periodic
(respectively, quasi periodic, Bohr almost periodic, almost
automorphic, recurrent, Levitan almost periodic, almost recurrent,
Poisson stable), where $\phi(t):=\mathcal L(\varphi(t))$ for any
$t\in\mathbb R$ and $\mathcal L(\varphi(t))\in \mathcal P(H)$ is
the law of random variable $\varphi(t)$.
\end{definition}

\begin{definition}\label{defS} Suppose that linear (homogeneous) equation
\begin{equation}\label{eqLE1}
x'=A(t)x
\end{equation}
generates a (linear) cocycle $\langle E,\varphi, (H(A),\mathbb
R,\sigma)\rangle$. We will say that linear homogeneous equation
(\ref{eqLE1}) (respectively, cocycle $\varphi$ generated by
equation (\ref{eqLE1})):
\begin{enumerate}
\item[-] satisfies condition (S) if for entire relatively compact
on $\mathbb R $ mild solution $\gamma$ of this equation we have
\begin{equation}\label{eqG1}
\lim\limits_{t\to +\infty}|\gamma(t)|=0
\end{equation}
(respectively, equality takes place for any relatively compact on
$\mathbb R$ mild solution of every equation
\begin{equation}\label{eqGB}
x'=B(t)x,
\end{equation}
where $B\in H(A)$);
\item is asymptotically stable if
\begin{equation}\label{eqG2}
\lim\limits_{t\to +\infty}|\varphi(t,x,A)|=0
\end{equation}
for any $x\in E$, where $\varphi(t,x,A):=U(t,A)x$ (respectively,
if equality $\lim\limits_{t\to +\infty}|\varphi(t,x,B)|=0$ takes
place for any $(x,B)\in E\times H(A)$.
\end{enumerate}
\end{definition}

\begin{lemma}\label{remS1} Cocycle $\varphi$ generated by equation (\ref{eqLE1}) satisfies condition
(S )if one of the following conditions are fulfilled:
\begin{enumerate}
\item the set $H(A)$ is compact and the cocycle $\varphi$
generated by equation (\ref{eqLE1}) is asymptotically stable;
\item equation (\ref{eqLE1}) is hyperbolic.
\end{enumerate}
\end{lemma}
\begin{proof} The first statement follows from the Theorem 2.37
\cite[ChII]{Che_2015}.

The second statement follows from Lemma \ref{lED_10}.
\end{proof}

Let $\varphi \in C(\mathbb R,H)$ be a solution of equation
(\ref{eqSDE}). Denote by
\begin{enumerate}
\item[-] $\mathfrak N_{\varphi}^{d}:=\{\{t_n\}:\ \varphi^{t_n}(t)$
converges in distribution to $\varphi(t)\}$ uniformly with respect
to $t$ on every compact from $\mathbb R$; \item[-] $\mathfrak
M_{\varphi}^{d}:=\{\{t_n\}:\ \varphi^{t_n}(t)$ converges in
distribution $\}$ uniformly with respect to $t$ on every compact
from $\mathbb R$.
\end{enumerate}

\begin{definition}\label{defS1} A solution $\varphi$ is said to be
compatible in distribution if $\mathfrak N_{(A,f,g)}\subseteq
\mathfrak N_{\varphi}^{d}$.
\end{definition}

\begin{theorem}\label{thS1} Suppose that the following conditions are
fulfilled:
\begin{enumerate}
\item[a.] $A\in C(\mathbb R,\Lambda)$ and equation (\ref{eqA})
generates a continuous cocycle $\langle E,\varphi, $ $(H(A),$
$\mathbb R,$ $\sigma)\rangle$ with fiber $E$ over base
$(H(A),\mathbb R,\sigma)$;

 \item[b.] equation (\ref{eqLE1}) satisfies to condition
(S);

\item[c.] the function $(A,f,g)\in C(\mathbb R,\Lambda)\times
C(\mathbb R,H)\times C(\mathbb R,H)$ is Poisson stable;

\item[d.] equation (\ref{eqSDE}) admits a solution $\varphi$
defined on $\mathbb R_{+}$ with precompact rang, i.e., the set
$Q_{+}:=\overline{\varphi(\mathbb R_{+})}$ is compact.
\end{enumerate}

Then equation (\ref{eqSDE}) has a unique solution $p$ defined on
$\mathbb R$ with precompact rang which is compatible.
\end{theorem}
\begin{proof} Under the conditions of Theorem \ref{thS1} equation
(\ref{eqSDE}) generates a continuous cocycle $\langle \mathcal
(H),\Phi, (H(A,f,g),\mathbb R,\sigma)\rangle$ with the fiber
$\mathcal P(H)$ over dynamical system $(H(A,f,g),\mathbb
R,\sigma)$ (see Theorem \ref{cocycle}). Denote by $\mu_0
:=\mathcal L(\varphi(0))$, then $\Phi(\mathbb R_{+},\mu_0,(A,$
$f,$ $g)):=\{\Phi(t,\mu_0,(A,f,g)):\ t\in\mathbb R_{+}\}$ is
precompact in $\mathcal P(H)$. Let $\gamma_1, \gamma_2\in
C(\mathbb R,H)$ be two arbitrary solutions of equation
(\ref{eqSDE}) with precompact ranges. Denote by
$\gamma(t):=\gamma_1(t)-\gamma_2(t)$, then $\gamma \in C(\mathbb
R,H)$ is a solution of equation (\ref{eqA}) with the precompact
range. Since equation (\ref{eqA}) satisfies the condition (S),
then we have equality (\ref{eqG1}). We will show that
\begin{equation}\label{eqF1}
\beta(\mathcal L(\gamma(t_n)),\mathcal L(\gamma_2(t_n)))\to 0
\end{equation}
as $n\to \infty$ for any $\{t_n\}\in\mathfrak N_{y_0}^{+\infty}$.
If we suppose that it is not true, then there exist
$\varepsilon_0>0$ and $\{t_n\}\in \mathfrak N_{y_0}^{+\infty}$
such that
\begin{equation}\label{eqF2}
\beta(\mathcal L(\gamma(t_n)),\mathcal L(\gamma_2(t_n)))\ge
\varepsilon_0
\end{equation}
for any $n\in\mathbb N$. Since the set $\gamma_{i}(\mathbb
R)\subset H$ ($i=1,2$) is precompct, then without loss of
generality we can suppose that the sequences $\{\gamma_{i}(t_n)\}$
($i=1,2$) are convergent. Denote by
$\bar{x}_{i}:=\lim\limits_{n\to \infty}\gamma_{i}(t_n)$. By
equality (\ref{eqF2}) we obtain $\bar{x}_1\not= \bar{x}_2$. Note
that $\gamma(t_n):=\gamma_1(t_n)-\gamma_2(t_n)\to
\bar{x}_1-\bar{x}_2\not= 0$ as $n\to \infty$. The last relation
contradicts to condition (S) (see item b.). The obtained
contradiction proves equality (\ref{eqF1}). To finish the proof it
is sufficient to apply Theorem \ref{t5.1} and Remark \ref{remC1}.
\end{proof}

\begin{coro}\label{cor_1} Under the conditions of Theorem
\ref{thS1} if $(A,f,g)$ is $\tau$-periodic (respectively, Levitan
almost periodic, almost recurrent, Poisson stable), then equation
(\ref{eqSDE}) has a unique solution $p$ defined on $\mathbb R$
which is $\tau$-periodic (respectively, Levitan almost periodic,
almost recurrent, Poisson stable).
\end{coro}
\begin{proof} This statement follows from Theorem \ref{thS1} and
Corollary \ref{cor3.2*}.
\end{proof}

\begin{theorem}\label{thS1_0} Under the condition of Theorem
\ref{thS1} if we replace condition (S) (see item b.) by asymptotic
stability of equation (\ref{eqA}), then equation (\ref{eqSDE}) has
a unique solution $p$ defined on $\mathbb R$ which is compatible
and $\lim\limits_{t\to +\infty}\beta(\mathcal L(p(t)),\mathcal
L(\varphi(t,x,$ $(A,$ $f,$ $g)))=0$.
\end{theorem}
\begin{proof} This statement can be proved using the same
arguments as in the proof of Theorem \ref{thS1} but instead of
Theorem \ref{t5.1} and Corollary \ref{cor3.2**} it is necessary to
apply Theorem \ref{t5.2} and Corollary \ref{cor3.2**}.
\end{proof}

\begin{coro}\label{cor_1I} Under the conditions of Theorem
\ref{thS1_0} if $(A,f,g)$ is $\tau$-periodic (respectively,
Levitan almost periodic, almost recurrent, Poisson stable), then
every solution of equation (\ref{eqSDE}) is asymptotically
$\tau$-periodic (respectively, asymptotically Levitan almost
periodic, asymptotically almost recurrent, asymptotically Poisson
stable) in distribution.
\end{coro}
\begin{proof} This statement follows from Theorem \ref{thS1_0} and
Corollary \ref{cor3.2_I}.
\end{proof}

\begin{definition}\label{defS2} A solution $\varphi$ is said to be
strongly compatible in distribution if $\mathfrak
M_{(A,f,g)}\subseteq \mathfrak M_{\varphi}^{d}$.
\end{definition}

\begin{theorem}\label{thS2} Suppose that the following conditions are
fulfilled:
\begin{enumerate}
\item[(i)] for any $B\in H(A)$ equation
\begin{equation}\label{eqLE2}
x'=B(t)x
\end{equation}
satisfies to condition (S); \item \item[(ii)] the set
$H(A,f,g)\subset C(\mathbb R,\Lambda)\times C(\mathbb R,H)\times
C(\mathbb R,H)$ is minimal; \item[(iii)] the function $(A,f,g)\in
C(\mathbb R,\Lambda)\times C(\mathbb R,H)\times C(\mathbb R,H)$ is
strongly Pois\-son stable; \item[(iv)] equation (\ref{eqSDE})
admits a solution $\varphi$ defined on $\mathbb R_{+}$ with
precompact rang.
\end{enumerate}

Then equation (\ref{eqSDE}) has a unique solution $p$ defined on
$\mathbb R$ with precompact rang which is strongly compatible in
distribution.
\end{theorem}
\begin{proof} To prove this statement we will use the same ideas
as in the proof of Theorem \ref{thS1}. Note that under the
conditions of Theorem \ref{thS2} equation (\ref{eqSDE}) generates
a continuous cocycle $\langle \mathcal (H),\Phi, (H(A,f,g),\mathbb
R,\sigma)\rangle$ with the fiber $\mathcal P(H)$ over dynamical
system $(H(A,f,g),\mathbb R,\sigma)$ (see Theorem \ref{cocycle}).
Denote by $\mu_0 :=\mathcal L(\varphi(0))$, then $\Phi(\mathbb
R_{+},\mu_0,(A,f,g)):=\{\Phi(t,\mu_0,(A,f,g)):\ t\in\mathbb
R_{+}\}$ is precompact in $\mathcal P(H)$. Let $\gamma_1,
\gamma_2\in C(\mathbb R,H)$ be two arbitrary solutions of some
equation (\ref{tilfg}) (where $(\tilde{A},\tilde{f},\tilde{g})\in
H(A,f,g)$) with precompact ranges. Denote by
$\gamma(t):=\gamma_1(t)-\gamma_2(t)$, then $\gamma \in C(\mathbb
R,H)$ is a solution of equation
\begin{equation}\label{eqtA}
x'=\tilde{A}(t)x
\end{equation}
with the precompact range. Since equation (\ref{eqtA}) satisfies
the condition (S), then we have equality (\ref{eqG1}). We will
show that
\begin{equation}\label{eqF1_1}
\beta(\mathcal L(\gamma(t_n)),\mathcal L(\gamma_2(t_n)))\to 0
\end{equation}
as $n\to \infty$ for any $\gamma_1,\gamma_2\in \Phi_{y}$, $y\in
H(y_0)$ and $\{t_n\}\in\mathfrak N_{y}^{+\infty}$. If we suppose
that it is not true, then there exist $q\in H(y_0)$,
$\varepsilon_0>0$ and $\{t_n\}\in \mathfrak N_{q}^{+\infty}$ such
that
\begin{equation}\label{eqF2_1}
\beta(\mathcal L(\gamma(t_n)),\mathcal L(\gamma_2(t_n)))\ge
\varepsilon_0
\end{equation}
for any $n\in\mathbb N$. Using the same arguments as in the proof
of Theorem \ref{thS1} we obtain a contradiction which proves our
statement. To finish the proof it is sufficient to apply Theorems
\ref{thB2} and \ref{t5.2} (see also Remark \ref{remC2}).
\end{proof}

\begin{coro}\label{cor_11} Under the conditions of Theorem
\ref{thS2} if $(A,f,g)$ is $\tau$-periodic (respectively, quasi
periodic, Bohr almost periodic, almost automorphic, recurrent,
strongly Poisson stable and $H(A,f,g)$ is a minimal set), then
equation (\ref{eqSDE}) has a unique solution $p$ defined on
$\mathbb R$ which is $\tau$-periodic (respectively, quasi
periodic, Bohr almost periodic, almost automorphic, recurrent,
strongly Poisson stable and $H(A,f,g)$ is a minimal set) in
distribution.
\end{coro}
\begin{proof} This statement follows from Theorem \ref{thS2} and
Corollary \ref{cor3.2**} (see also Corollary \ref{cor10B}).
\end{proof}

\begin{theorem}\label{thS2_0} Under the condition of Theorem
\ref{thS2} if we replace condition (S) (see item (i).) by
asymptotic stability of equation of every equation
(\ref{eqSDE_B2}), then equation (\ref{eqSDE}) has a unique
solution $p$ defined on $\mathbb R$ which is strongly compatible
in distribution and $\lim\limits_{t\to +\infty}\beta(\mathcal
L(p(t)),\mathcal L(\varphi(t,x,(A,f,g)))=0$.
\end{theorem}
\begin{proof} This statement can be proved using the same
arguments as in the proof of Theorem \ref{thS2} but instead of
Theorem \ref{t5.1} and Corollary \ref{cor3.2**} it is necessary to
apply Theorem \ref{t5.2} and Corollary \ref{cor3.2**}.
\end{proof}

\begin{coro}\label{cor_S2} Under the conditions of Theorem
\ref{thS2_0} if $(A,f,g)$ is $\tau$-periodic (respectively, quasi
periodic, Bohr almost periodic, almost automorphic, recurrent,
strongly Poisson stable and $H(A,f,g)$ is a minimal set), then
every solution of equation (\ref{eqSDE}) is asymptotically
$\tau$-periodic (respectively, quasi periodic, Bohr almost
periodic, almost automorphic, recurrent, strongly Poisson stable
and $H(A,f,g)$ is a minimal set) in distribution.
\end{coro}
\begin{proof} This statement follows from Theorem \ref{thS2_0} and
Corollary \ref{cor3.2**}.
\end{proof}

\end{document}